\documentclass[a4paper, 10pt, parskip=half]{scrartcl}
\usepackage[utf8]{inputenc}
\usepackage[T1]{fontenc}
\usepackage{lmodern}
\usepackage{csquotes}
\usepackage{amsmath}
\usepackage{amssymb}
\usepackage{amsthm}
\usepackage{amsfonts}
\usepackage{dsfont}
\usepackage{mathtools}
\usepackage{bbm}
\usepackage{marginnote}
\usepackage{caption}
\usepackage{subcaption}
\usepackage{tikz}
\usepackage{hyperref}
\hypersetup{%
  colorlinks=true,
  linkcolor=black,
  citecolor=black,
  urlcolor=blue,
  pdftitle={Absence of critical mass phenomena in one-dimensional critical quasilinear Keller--Segel systems},
  pdfauthor={Xinru Cao and Mario Fuest},
  pdfkeywords={},
  bookmarksopen=true,
}

\allowdisplaybreaks

\usepackage[numbers, compress]{natbib}

\newcommand{\R}{\mathbb{R}}

\newcommand{\N}{\mathbb{N}}

\newcommand{\mc}[1]{\mathcal{#1}}
\newcommand{\ur}[1]{\mathrm{#1}}
\newcommand{\ure}{\ur{e}}

\ifdefined\labelenumi%
  \renewcommand{\labelenumi}{(\roman{enumi})}
  
\fi

\newcommand{\eps}{\varepsilon}

\DeclareMathOperator{\id}{id}

\newcommand{\defs}{\coloneqq}
\newcommand{\sfed}{\eqqcolon}

\newcommand{\nea}{\nearrow}

\newcommand{\ol}{\overline}

\newcommand{\dx}{\,\mathrm{d}x}
\newcommand{\ds}{\,\mathrm{d}s}

\newcommand{\dsigma}{\,\mathrm{d}\sigma}
\newcommand{\drho}{\,\mathrm{d}\rho}

\newcommand{\ddt}{\frac{\mathrm{d}}{\mathrm{d}t}}

\newcommand{\embed}{\hookrightarrow}

\newcommand{\hp}{\hphantom}
\newcommand{\pe}{\mathrel{\hp{=}}}

\newcommand{\tmax}{T_{\max}}

\newcommand{\intom}{\int_0^1}

\newcommand{\intnstom}{\int_0^t \int_0^1}

\newcommand{\Ombar}{[0, 1]}
\newcommand{\loc}{\mathrm{loc}}

\newcommand{\leb}[2][(0, 1)]{\ensuremath{L^{#2}(#1)}}

\newcommand{\sob}[3][(0, 1)]{\ensuremath{W^{#2, #3}(#1)}}

\newcommand{\con}[2][\Ombar]{\ensuremath{C^{#2}(#1)}}

\newcommand{\nn}{\nonumber}

\newcommand{\norm}[2][ ]{\|#2\|_{#1}}

\makeatletter
\renewenvironment{proof}[1][\proofname]{\par
  \pushQED{\qed}%
  \normalfont \topsep0\p@\relax
  \trivlist
  \item[\hskip\labelsep\scshape
  #1\@addpunct{.}]\ignorespaces
}{%
  \popQED\endtrivlist\@endpefalse
}
\makeatother

\newtheorem{base}{Base}[section]
\numberwithin{equation}{section}

\newtheorem{theorem}[base]{Theorem} \newtheorem*{theorem*}{Theorem}
\newtheorem{lemma}[base]{Lemma} \newtheorem*{lemma*}{Lemma}
\newtheorem{prop}[base]{Proposition} \newtheorem*{prop*}{Proposition}
 \newtheorem*{cor*}{Corollary}

\theoremstyle{definition}
\newtheorem{remark}[base]{Remark} \newtheorem*{remark*}{Remark}
 \newtheorem*{definition*}{Definition}
 \newtheorem*{example*}{Example}
 \newtheorem*{cond*}{Condition}

\setkomafont{title}{\normalfont\Large}
\title{Absence of critical mass phenomena in one-dimensional critical quasilinear Keller--Segel systems}
\usepackage{authblk}

\author[1]{Xinru Cao\footnote{e-mail: caoxinru@gmail.com}}
\author[2,3]{Mario Fuest\footnote{e-mail: fuest@ifam.uni-hannover.de, corresponding author}}

\affil[1]{School of Mathematics and Statistics, Donghua University, North Renmin Road 2999, 201620 Shanghai, China}
\affil[2]{Leibniz Universität Hannover, Institut für Angewandte Mathematik, Welfengarten 1, 30167 Hannover, Germany}
\affil[3]{Institut für Mathematik, Universität Kassel, Heinrich-Plett Str.\ 40, 34132 Kassel, Germany}

\begin{document}
\date{}
\maketitle

\KOMAoptions{abstract=true}
\begin{abstract}
\noindent
We consider the Neumann initial boundary value problem associated to the chemotaxis system
\begin{align}\label{prob:abstract}\tag{$\star$}
  \begin{cases}
    u_t = \big((u+1)^{m-1} u_x - u(u+1)^m v_x\big)_x & \text{in $(0, 1) \times (0, \infty)$}, \\
    v_t = v_{xx} - v + u,  &\text{in $(0, 1) \times (0, \infty)$},
  \end{cases}
\end{align}
where $m \in \mathbb R$ is a given parameter.
The relation between diffusion and taxis sensitivity is critical since the ratio $u(u+1)^m/(u+1)^{m-1}$ grows like $u^{2/n}$ for large $u$ with $n = \dim((0, 1)) = 1$.

Nonetheless, we show that there is no critical mass phenomenon if $m \le -1$; that is, in that case all solutions emanating from suitably regular initial data are globally bounded.
For certain parabolic--elliptic simplifications of \eqref{prob:abstract}, we obtain the same conclusion for all $m \in (-\infty, -1] \cup (0, \infty)$
and even for all $m \in \mathbb R$ if the initial datum is additionally assumed to be monotone.

This stands in contrast to critical mass phenomena known to occur for critical quasilinear Keller--Segel systems considered in higher-dimensional domains.
Accordingly, we make use of several special features of the one-dimensional setting
such as the boundedness of the energy functional from below, the embedding $W^{1, n} \hookrightarrow L^\infty$,
and the fact that the mass accumulation function solves a spatially non-degenerate parabolic equation. \\
\textbf{Key words:} {chemotaxis; critical mass; boundedness; quasilinear parabolic problems} \\
\textbf{AMS Classification (2020):} {35B33 (primary); 35B45, 35B65, 35K59, 92C17 (secondary)}
\end{abstract}

\section{Introduction}
The quasilinear Keller--Segel model 
\begin{align}\label{quasilinear}
  \begin{cases}
    u_t = \nabla\cdot\big(D(u)\nabla u - S(u)\nabla v\big) & \text{in } \Omega\times(0,T),\\
    \tau v_t = \Delta v-v+u  & \text{in } \Omega\times(0,T),
  \end{cases}
\end{align}
describes the migration of cells in response to their surrounding environment in a smooth, bounded domain $\Omega\subset \mathbb{R}^n$ ($n\ge 1$).
It has been proposed by Hillen and Painter  in \cite{PainterHillenVolumefillingQuorumsensingModels2002}
(for certain choices of $D$ and $S$; see also the pioneering work \cite{KellerSegelInitiationSlimeMold1970} for $D \equiv 1$ and $S = \id$
as well as \cite{HillenPainterUserGuidePDE2009} for an overview of related models).
Here, $u$ and $v$ stand for the density of cells and the concentration of a chemical signal produced by the cells, respectively.
Moreover, $\tau \in \{0, 1\}$ is a given parameter, and the diffusion rate $D \ge 0$ and the chemotactic sensitivity $S \ge 0$ represent
the strength of random and directed motion, respectively, with respect to the local cell density.

In particular depending on the behaviour of the quotient $\tfrac{S(\sigma)}{D(\sigma)}$ as $\sigma\nearrow \infty$,
different solution dynamics can be observed.
For the sake of exposition, we focus on the prototypical choices
\begin{align}\label{DS_gen}
  D(\sigma)=(\sigma+1)^{m-1}  \quad \text{ and } \quad S(\sigma)=\sigma(\sigma+1)^{q-1}
  \qquad \text{for } \sigma \ge 0
\end{align}
for $m,q\in\mathbb{R}$.

If $q-(m-1)<\frac{2}{n}$ and $n\ge 1$, solutions of \eqref{quasilinear} are always global and bounded (\cite{TaoWinklerBoundednessQuasilinearParabolic2012}, \cite{IshidaEtAlBoundednessQuasilinearKeller2014}, \cite{DingWinklerRadialBlowupQuasilinear2024}, see also \cite{HorstmannWinklerBoundednessVsBlowup2005}, \cite{SenbaSuzukiQuasilinearParabolicSystem2006}).
On the other hand, if $q-(m-1)>\frac{2}{n}$ with $n\ge 2$, unbounded solutions have been constructed in \cite{WinklerDoesVolumefillingEffect2009} (see also \cite{HorstmannWinklerBoundednessVsBlowup2005} and \cite{DingWinklerRadialBlowupQuasilinear2024}) without identifying whether blow-up happens at finite or infinite time.
If additionally $q\le 0$, i.e., if $S$ is bounded, solutions are always global which excludes the possibility of finite-time blow-up
(\cite{WinklerGlobalClassicalSolvability2019}, \cite{LankeitInfiniteTimeBlowup2020}).
If $q>0$ (and still $q-(m-1)>\frac2n$ as well as $n \ge 2$), however,
finite-time blow-up is known to occur in the parabolic--elliptic setting $\tau = 0$ (\cite{WinklerDjieBoundednessFinitetimeCollapse2010})
and under certain additional assumptions also in the fully parabolic setting $\tau = 1$
(\cite{CaoFuestFinitetimeBlowupFully2025}, see also \cite{CieslakStinnerFinitetimeBlowupGlobalintime2012}, \cite{CieslakStinnerFiniteTimeBlowupSupercritical2014}, \cite{CieslakStinnerNewCriticalExponents2015}).
Regarding finite time blow-up for $q-(m-1) > \frac2n$ and $n = 1$, we refer to \cite{cieslak2010looking} and \cite{CieslakLaurencotFiniteTimeBlowup2010}
(and also to \cite{CieslakWinklerFinitetimeBlowupQuasilinear2008}).

A natural next step towards investigating the precise strength of the destabilising chemotaxis term consists of considering the case of critical exponents;
that is, to assume
\begin{align}\label{critical_exponent}
  q-(m-1)=\frac 2n.
\end{align}
For $n\ge 2$, this leads to \emph{critical mass phenomena}: For initial data with $\int_\Omega u_0$ below a certain threshold, all solutions are global and bounded,
while for all larger $M$ there are initial data with $\int_\Omega u_0 = M$ leading to blow-up.
Most famously, this has been observed for the minimal Keller--Segel system (obtained upon the choices $m=q=1$) in two-dimensional settings,
where the critical mass is $4\pi$ for arbitrary smooth domains
(\cite{NagaiEtAlApplicationTrudingerMoser1997}, \cite{GajewskiZachariasGlobalBehaviourReactiondiffusion1998}, \cite{NagaiBlowupNonradialSolutions2001}, \cite{HorstmannWangBlowupChemotaxisModel2001}),
$8\pi$ in the radially symmetric setting
(\cite{NagaiEtAlApplicationTrudingerMoser1997}, \cite{GajewskiZachariasGlobalBehaviourReactiondiffusion1998}, \cite{MizoguchiWinklerBlowupTwodimensionalParabolic},
see also \cite{MizoguchiFinitetimeBlowupCauchy2020}, \cite{FuestEtAlBlowupFullyParabolic2026})
as well as for the full space (\cite{BlanchetEtAlTwodimensionalKellerSegelModel2006})
and $2\pi$ for the quarter disc (\cite{FuestLankeitCornersCollapseSimple2023}).
For generalisations to quasilinear systems with more general $m$ and $q$ lying on the critical line \eqref{critical_exponent},
see \cite{LaurencotMizoguchiFiniteTimeBlowup2017}, \cite{winkler2022family}, \cite{CaoGaoCriticalMassQuasilinear2023} and also \cite{DingWinklerRadialBlowupQuasilinear2024}.

To the best of our knowledge, the only example of parameters fulfilling \eqref{critical_exponent} for which there is \emph{no} critical mass phenomenon
is $n = 1$ and $(m, q) = (0, 1)$.
Indeed, in that case solutions are known to be always globally bounded,
also for initial data with large mass (\cite{cieslak2018no}, \cite{CieslakFujieGlobalExistence1D2020}, \cite{BieganowskiEtAlBoundednessSolutionsCritical2019}).
The corresponding proofs are based on an energy functional involving spatial derivatives of $u$,
and attempts to extend these to other critical $(m, q)$ have so far not resolved the question of global boundedness.
(See \cite{CieslakEtAlNonlinearFisherInformation2025} for recent developments in this direction).

These results motivate the question:
\begin{center}
  \emph{Is the absence of a critical mass distinguishing between boundedness and blow-up \\
  limited to the particular choice $n = 1$, $(m, q) = (0, 1)$ or \\
  a general phenomenon of critical quasilinear Keller--Segel systems in 1D?}
\end{center}

To address it, we study the problem
\begin{align}\label{system}
  \begin{cases}
    u_t = \big(D(u) u_x - S(u)v_x\big)_x            & \text{in $(0, 1) \times (0, T)$}, \\
    \tau v_t = v_{xx} - k v - (1-k) M + u           & \text{in $(0, 1) \times (0, T)$}, \\
    u_x = v_x = 0                                   & \text{on $\{0, 1\} \times (0, T)$}, \\
    u(\cdot, 0) = u_0, \tau v(\cdot, 0) = \tau v_0  & \text{in $(0, 1)$},
  \end{cases}
\end{align}
where
$M = \int_0^1 u_0$, $\tau, k \in \{0, 1\}$ as well as
\begin{align}\label{DS}
  D(\sigma) = (\sigma+1)^{m-1}, \quad S(\sigma) = \sigma (\sigma+1)^m \qquad \text{for all } \sigma\ge 0
\end{align}
and some $m \in \R$.
That is, we consider the fully parabolic ($(\tau,k)=(1,1)$), the parabolic--elliptic ($(\tau,k)=(0,1)$) and the Jäger--Luckhaus ($(\tau,k)=(0,0)$) settings.
(The choice $(\tau, \kappa) = (1, 0)$ could be treated in a similar manner as $(\tau, \kappa) = (1, 1)$,
but as it appears to be rather unnatural and some small modifications would be necessary, we choose to not cover it here.)
We note that \eqref{DS} corresponds to \eqref{DS_gen} with $q = m+1$, which according to \eqref{critical_exponent} is the critical case for $n = 1$.

Our goal is to show that solutions of \eqref{system} emanating from arbitrary initial data
\begin{align}\label{initial_data}
  u_0, \tau v_0 \in \bigcup_{\theta > 1} W^{1, \theta}((0, 1)) \quad \text{with} \quad 0 \not \equiv u_0 \ge 0,\, \tau v_0 \ge 0
\end{align}
remain bounded throughout evolution, i.e., that there is no critical mass with respect to global boundedness.

\paragraph{A criterion for global boundedness.}
We first establish a refined criterion for global existence and boundedness of solutions to \eqref{system}.
\begin{prop}\label{prop:eps_reg}
  Let $(\tau, k) \in \{(0, 0), (0, 1), (1, 1)\}$ and $m \in \R$.
  For all initial data $u_0, v_0$ fulfilling \eqref{initial_data},
  there exist $\tmax \in (0, \infty]$ and a uniquely determined, nonnegative solution
  \begin{align}\label{eq:eps_reg:reg}
    (u, v) \in \bigcup_{\theta > 1} C^0\big([0, \tmax); W^{1, \theta}((0, 1))\big) \cap C^{\infty}\big([0, 1] \times (0, \tmax)\big)
    \quad \text{with} \quad
    (1-k) \intom v = 0 
  \end{align}
  of \eqref{system} such that $\limsup_{t \nea \tmax} \|u(\cdot, t)\|_{L^\infty((0, 1))} = \infty$ if $\tmax < \infty$.

  Moreover, there is $\eps > 0$ such that if $u_0, v_0$ satisfy \eqref{initial_data}
  and the corresponding solution $(u, v)$ of \eqref{system} is such that there is $\delta > 0$ with the property that
  \begin{align}\label{eq:eps_reg:eps_cond}
    \int_a^b u(\cdot, t) < \eps \qquad \text{for all $t \in (0, \tmax)$ whenever $0 \le a \le b \le 1$ are such that $|b-a| \le \delta$,}
  \end{align}
  then $(u, v)$ is bounded in $(0, 1) \times (0, \tmax)$ (and hence $\tmax = \infty$).
\end{prop}

\begin{remark}
  \begin{enumerate}
    \item
      Results of this flavour are known as $\varepsilon$-regularity theorems.
      Often these are stated in a more local fashion; that is, instead of asserting that an analogous version of \eqref{eq:eps_reg:eps_cond}
      implies global boundedness whenever $\eps > 0$ is small enough,
      it is claimed that an analogue of $\sup_{t \in (0, \tmax)} \int_{x_0-\delta/2}^{x_0+\delta/2} u(\cdot, t) < \eps$
      implies boundedness of $u$ near $x_0$ for small $\eps > 0$.
      (Here, “analogue” inter alia means that $L^1$ is replaced by a space critical for the problem at hand.)
      We refer to \cite{sugiyama2010varepsilon} (and also \cite{sugiyama2009varepsilon}) for the quasilinear Keller--Segel model posed in the higher dimensional full space
      and to the seminal work \cite{caffarelli1982partial} for the Navier--Stokes equations, for instance.

    \item
      Some conditions implying \eqref{eq:eps_reg:eps_cond} are listed in Proposition~\ref{prop:eps_reg_alt_cond} below.

    \item
      For the minimal Keller--Segel system considered on $n$-dimensional domains,
      uniform integrability of $\{\, u^\frac{n}{2}(\cdot, t) \mid t \in (0, \tmax)\}$ is known to yield global boundedness (\cite[Chapter~2]{CaoGlobalSolutionsChemotaxis2017}).
      In an analogous fashion, we could also obtain global boundedness by requiring, instead of \eqref{eq:eps_reg:eps_cond},
      that $\{\, u(\cdot, t) \mid t \in (0, \tmax)\,\}$ is uniformly integrable on $(0, 1)$  --
      but that would be unsuitable for our proof of Theorem~\ref{th:gb_pe} below;
      see Proposition~\ref{lm:ex_w_hoelder_u_not_equiint} below and the discussion directly beforehand for more details.
  \end{enumerate}
\end{remark}

Evidently, Proposition~\ref{prop:eps_reg} implies (upon taking $\delta = 1$ in \eqref{eq:eps_reg:eps_cond} and making use of mass conservation, i.e., of \eqref{eq:u_mass})
that for all initial data with sufficiently small mass, the corresponding solution of \eqref{system} is globally bounded.
\begin{prop}\label{prop:gb_small_mass}
  Let $(\tau, k) \in \{(0, 0), (0, 1), (1, 1)\}$ and $m > 0$.
  There exists $\eps > 0$ such that if $u_0$ and $\tau v_0$ fulfil \eqref{initial_data} as well as $\int_0^1 u_0 < \eps$,
  then the solution $(u, v)$ of \eqref{system} obtained in Proposition~\ref{prop:eps_reg} is global and bounded uniformly in time.
\end{prop}
Our main results stated below do not require smallness of $\intom u_0$, however.

\paragraph{The parabolic--elliptic case with positive $m$.}
As a first nontrivial application of Proposition~\ref{prop:eps_reg}, we consider $\tau = 0$, $k \in \{0, 1\}$ and $m > 0$.
The key observation is that the accumulated mass function $w$ defined by $w(x, t) = \int_0^x u(\rho, t) \drho$ then solves an equation of $p$-Laplace type,
which, unlike its counterparts arising in higher-dimensional settings, is not degenerate in the spatial variable.
Since $v_x$ is bounded (see Lemma~\ref{lm:vx_linfty}) and the relation between $S$ and $D$ is critical,
well-known Hölder regularity results (cf.\ \cite{DiBenedettoDegenerateParabolicEquations1993}) become applicable for $w$ (see Lemma~\ref{lm:w_hoelder}).
As these imply \eqref{eq:eps_reg:eps_cond} by Proposition~\ref{prop:eps_reg_alt_cond}(iii), we obtain the following.
\begin{theorem}\label{th:gb_pe}
  Let $(\tau,k) \in \{ (0,0),(0,1)\}$, $m > 0$ and suppose \eqref{initial_data}.
  Then the solution $(u, v)$ of \eqref{system} obtained in Proposition~\ref{prop:eps_reg} is global and bounded uniformly in time.
\end{theorem}

\paragraph{The Jäger--Luckhaus case with monotone initial data.}
For nonpositive $m$, a corresponding regularity theory for the $w$-equation does not seem to be available.
After a J\"ager--Luckhaus transformation, we first observe in Lemma~\ref{lm:w_le_x_alpha} that $\ol w(x, t) \defs C x^{\alpha}$ is a supersolution for sufficiently small but positive $\alpha$, allowing us to essentially recover Hölder continuity of $w$ in $0$.
For estimates away from the origin, we recall in Lemma~\ref{lm:u_monotone} that if we not only assume $\tau = 0$ but also $k = 0$,
then spatial monotonicity of $u$ is preserved throughout evolution.
Since the mass is preserved, this yields uniform-in-time $L_{\loc}^\infty((0, 1])$ estimates for $u$ (see the proof of Lemma~\ref{lm:u_lp_jl}).
In combination, these bounds render Proposition~\ref{prop:eps_reg} applicable, implying:
\begin{theorem}\label{th:gb_jl}
  Let $(\tau,k) = (0,0)$ and $m \in \R$.
  For all nonincreasing $u_0 \colon \Ombar \to [0, \infty)$ satisfying \eqref{initial_data}, the solution $(u, v)$ of \eqref{system} obtained in Proposition~\ref{prop:eps_reg} is global and bounded uniformly in time.
\end{theorem}

\begin{remark}
  By symmetry, Theorem~\ref{th:gb_jl} also applies to nondecreasing $u_0$ as well as to $u_0$ which are symmetric around the axis $x=\frac12$ and monotone in $(0, \frac12)$, for instance.
\end{remark}

\paragraph{The fully parabolic case with bounded $S$.}
Finally, we turn to the fully parabolic setting, i.e., to the case $(\tau,k)=(1, 1)$.
Solutions to \eqref{system} are always global if $S$ is bounded and $D$ decays at most algebraically,
as can be rapidly seen by extending the approach in \cite{WinklerGlobalClassicalSolvability2019},
based on testing procedures, maximal Sobolev regularity and Moser iteration, to the one-dimensional setting.
(See also \cite[Theorem~5.1]{CieslakEtAlNonlinearFisherInformation2025} for an alternative proof when $D$ and $S$ fulfil \eqref{DS} for $m = -1$).
While this result covers both supercritical and critical cases,
it leaves the question open whether, for bounded $S$ and corresponding critical $D$ as in \eqref{DS},
a critical mass distinguishing between global boundedness and the occurrence of infinite-time blow-up exists.

Our final main result provides a negative answer.
Its proof rests on special properties of the one-dimensional setting in two places:
First, we note that the Lyapunov functional defined in \eqref{eq:energy:def_F} remains bounded from below along trajectories
and that hence its dissipation rate is integrable in time (see Lemma~\ref{lm:vt_l2_spacetime}).
Second, the embedding $\sob11 \embed \leb\infty$ plays a key role in deriving suitable estimates for another functional in Lemma~\ref{lm:ddt_psi_u}.
Combined with the assumption that $S$ is bounded, these preparations yield a uniform-in-time $L\log L$ bound for $u$ in Lemma~\ref{lm:ulnu} and hence the following.
\begin{theorem}\label{th:gb_pp}
  Let $(\tau, k) \in \{(0, 0), (0, 1), (1, 1)\}$, let $m \le -1$ and suppose \eqref{initial_data}.
  Then the solution $(u, v)$ of \eqref{system} obtained in Proposition~\ref{prop:eps_reg} is global and bounded uniformly in time.
\end{theorem}

\paragraph{Plan of the paper.}
We recall local existence theory and prove Proposition~\ref{prop:eps_reg} in Section~\ref{sec:eps_reg},
introduce the mass accumulation function $w$, discuss pointwise estimates of $v_x$ and derive Theorem~\ref{th:gb_pe} in Section~\ref{sec:gb_pe},
before showing Theorem~\ref{th:gb_jl} in Section~\ref{sec:gb_jl},
and then obtain a lower bound for the energy $\mc F(u, v)$ and prove Theorem~\ref{th:gb_pp} in Section~\ref{sec:gb_pp}.

\section{Local existence and refined criteria for global boundedness: proof of Proposition~\ref{prop:eps_reg}}\label{sec:eps_reg}
We begin by noting that \eqref{system} is locally well-posed and by stating the standard criterion for global existence.
\begin{lemma}\label{lm:local_ex}
  Let $(\tau, k) \in \{(0, 0), (0, 1), (1, 1)\}$, let $m \in \R$ and suppose \eqref{initial_data}.
  Then there exist $T_{\max}\in(0,\infty]$ and a uniquely determined classical solution $(u, v)$ of \eqref{system} with regularity \eqref{eq:eps_reg:reg}
  being such that $T_{\max}<\infty$ implies $\limsup_{t \nea \tmax} \|u(\cdot, t)\|_{L^\infty((0, 1))} = \infty$.

  Moreover, $u$ and $v$ are positive in $[0, 1] \times (0, \tmax)$,
  and
  \begin{align}\label{eq:u_mass}
    \int_0^1 u(\cdot, t) = M
    \quad \text{as well as} \quad
    \int_0^1 v(\cdot, t) \le \max\left\{M, \intom \tau v_0 \right\}
    \qquad \text{hold for all $t \in (0, \tmax)$}.
  \end{align}
\end{lemma}
\begin{proof}
  Local solvability follows from \cite[Theorem~14.4, Theorem~14.6 and Corollary~14.7]{AmannNonhomogeneousLinearQuasilinear1993} for the fully parabolic case
  and similarly as in \cite[Section~1]{CieslakWinklerFinitetimeBlowupQuasilinear2008} or \cite[Theorem~2.1]{TelloWinklerChemotaxisSystemLogistic2007} for the parabolic--elliptic setting.
  The strong maximum principle implies positivity and integrating the first and second equation in \eqref{system} yields \eqref{eq:u_mass}.
\end{proof}

That the interplay between diffusion and cross-diffusion in \eqref{system} is critical is indicated by a scaling-related argument:
When considering the temporal derivative of $\intom u^p$, worrisome terms can barely (or barely not) be controlled by dissipative ones
if one makes direct use of the Gagliardo--Nirenberg inequality which asserts that for all $\alpha > 2$ there is $C_\alpha > 0$ such that
\begin{align}\label{eq:gni}
  \int_0^1 |\varphi|^{\alpha}
  \le C_\alpha \left( \int_0^1 \big|(|\varphi|^{\frac{\alpha-2}{2}})_x\big|^2 \right) \left(\int_0^1 |\varphi|\right)^2 + C_\alpha \left(\int_0^1 |\varphi|\right)^\alpha
\end{align}
for all $\varphi \in L^1((0, 1))$ with $|\varphi|^\frac{\alpha-2}{2} \in W^{1, 2}((0, 1))$.
If $M = \intom u$ is sufficiently small, the coefficient of the gradient term on the right-hand side of \eqref{eq:gni} applied to $\varphi = u$ becomes small as well --
which allows for the desired absorption of the most problematic term stemming from cross-diffusion.
While such a line of reasoning is able to prove Proposition~\ref{prop:gb_small_mass}, it does not provide any information for large $M$.

Instead, the following refinement of \eqref{eq:gni} will be a crucial ingredient in the proof of Proposition~\ref{prop:eps_reg}.
\begin{lemma}\label{lm:eps_inq}
  Let $\alpha > 2$ and $\eta > 0$.
  There exists $\eps > 0$ such that for all $\delta > 0$ there is $C > 0$ with the following property:
  If $\varphi\in L^1((0,1))$ satisfies $|\varphi|^{\frac{\alpha-2}{2}}\in W^{1,2}((0,1))$ as well as
  \begin{align}\label{eq:esp_inq}
    \int_a^b |\varphi| < \eps
    \qquad \text{for all $(a, b) \subset (0, 1)$ with $|b-a| < \delta$},
  \end{align}
  then
   \begin{align}\label{eq:interpolation}
     \int_0^1 |\varphi|^{\alpha} \le
     \eta \int_0^1 \big|(|\varphi|^{\frac{\alpha-2}{2}})_x\big|^2 + C.
  \end{align}
\end{lemma}
\begin{proof}
  We first make use of the Gagliardo--Nirenberg inequality to find $c_1, c_2 > 0$ such that
  \begin{align}\label{eq:eps_reg_lp:gni}
        \int_a^b |\psi|^\frac{2\alpha}{\alpha-2}
    \le c_1 \left( \int_a^b |\psi_x|^2 \right) \left( \int_a^b |\psi|^\frac{2}{\alpha-2} \right)^2 + c_2 |b-a|^{1-\alpha} \left( \int_a^b |\psi|^\frac{2}{\alpha-2} \right)^{\alpha}
  \end{align}
  for all $\psi\in W^{1, 2}((a, b))$ and all $(a, b) \subseteq (0, 1)$.
  By scaling and translating, we see that $c_1$ and $c_2$ do not depend on $(a, b)$.

  We now set $\eps \defs \min\{(\frac{\eta}{c_1})^\frac12, 1\}$ and fix $\delta>0$ according to \eqref{eq:esp_inq}.
  After shrinking $\delta$, if necessary, we may assume $\delta^{-1} \in \N$.
  Thus, by splitting $(0, 1)$ in intervals of length $\delta$ and choosing $\psi=|\varphi|^\frac{\alpha-2}{2}$ in \eqref{eq:eps_reg_lp:gni}, we obtain
  \begin{align*}
          \int_0^1 |\varphi|^{\alpha}
    &=    \sum_{j=0}^{\delta^{-1}-1} \int_{j\delta}^{j\delta + \delta} \left(|\varphi|^\frac{\alpha-2}{2}\right)^\frac{2\alpha}{\alpha-2} \notag \\
    &\le  c_1 \eps^2 \sum_{j=0}^{\delta^{-1}-1} \int_{j\delta}^{j\delta + \delta} \left(|\varphi|^\frac{\alpha-2}{2}\right)_x^2 + c_2 \delta^{-\alpha} \eps^\alpha \notag \\
    &\le  \eta \int_0^1 \left(\varphi^\frac{\alpha-2}{2}\right)_x^2 + c_2 \delta^{-\alpha}
  \end{align*}
  and hence \eqref{eq:interpolation} for $C \defs c_2 \delta^{-\alpha}$.
\end{proof}

Next, we state a time-weighted version of maximal Sobolev regularity, which allows us to bridge the cross-diffusion term with the cell density $u$ via the second equation in \eqref{system}.
Related estimates have already been used for deriving uniform boundedness for certain fully parabolic chemotaxis systems
(see, for instance, \cite{CaoLargeTimeBehavior2017}, \cite{caotao2021boundedness}).
However, we choose to provide a proof here in order to emphasise that the constant on the right-hand side of \eqref{eq:max_reg:est} below is independent of $\gamma$.
\begin{lemma}\label{lm:max_reg}
  Let $(\tau, k) \in \{(0, 0), (0, 1), (1, 1)\}$ and let $m \in \R$.
  For all $r \in (1, \infty)$, there exists $C > 0$ such that for all $\gamma \in [0, \frac1{2r}]$ and all initial data fulfilling not only \eqref{initial_data} but also $\tau v_0 \in \sob2r$ with $\partial_\nu (\tau v_0) = 0$ on $\partial \Omega$,
  the solution $(u, v)$ of \eqref{system} given by Lemma~\ref{lm:local_ex} fulfils
  \begin{align}\label{eq:max_reg:est}
        \intnstom \ure^{\gamma r s} |v_{xx}(x, s)|^r \dx \ds
    \le C \intnstom \ure^{\gamma r s} (u(x, s))^r \dx \ds
        + C \|\tau v_0\|_{\sob2r}^r
  \end{align}
  for all $t \in (0, \tmax]$.
\end{lemma}
\begin{proof}
  If $\tau = 0$, elliptic regularity (cf.\ \cite[Theorem~19.1]{FriedmanPartialDifferentialEquations1976}) yields $C > 0$ such that $\int_0^1 |v_{xx}|^r \le C \int_0^1 u^r$, which already yields the desired estimate.
  (In the considered one-dimensional setting, we could alternatively directly use the identity $|v_{xx}|^r = |k v + (1-k) M - u|^r$ and, if $k = 1$, test the second equation in \eqref{system} with $v^{r-1}$ to obtain $\int_0^1 v^r \le C_r \int_0^1 u^r$.)

  For $(\tau, k) = (1, 1)$, we note that the function $z$ defined by $z(x, t) \defs \ure^{\gamma t} v(x, t)$ for $x \in (0, 1)$ and $t \in (0, \tmax)$ solves
  \begin{align*}
    \begin{cases}
      z_t = z_{xx} - \frac12 z + (\gamma - \frac12) \ure^{\gamma t} v + \ure^{\gamma t} u  & \text{in $(0, 1) \times (0, \tmax)$}, \\
      z_x = 0 & \text{on $\{0, 1\} \times (0, \tmax)$}, \\
      z(\cdot, 0) = v_0 & \text{in $(0, 1)$}
    \end{cases}
  \end{align*}
  According to \cite[Theorem~3.1]{HieberPrussHeatKernelsMaximal1997}, there is $c_1 > 0$ such that
  \begin{align*}
        \intnstom |z_{xx}(x, s)|^r \dx \ds
    \le c_1 \intnstom \ure^{\gamma r s} |(\gamma - \tfrac12) v(x, s) + u (x, s)|^r \dx \ds
        + c_1 \|v_0\|_{\sob2r}^r
  \end{align*}
  for all $t \in (0, \tmax]$.
  As moreover
  \begin{align*}
          \ddt \left( \ure^{\gamma r t} \int_0^1 v^r \right)
    &=    -(r-1) \int_0^1 v^{r-2} |v_x|^2
          -(1-\gamma r) \ure^{\gamma r t} \int_0^1 v^r
          + \ure^{\gamma r t} \int_0^1 u^r \\
    &\le  -\frac12 \ure^{\gamma r t} \int_0^1 v^r
          + \ure^{\gamma r t} \int_0^1 u^r
    \qquad \text{in $(0, \tmax)$}
  \end{align*}
  we may estimate
  \begin{align*}
        \intnstom \ure^{\gamma r s} (v(x, s))^r
    \le 2 \intnstom \ure^{\gamma r s} (u(x, s))^r
        + \int_0^1 v_0^r
  \end{align*}
  and hence (due to $|\gamma - \frac12| \le \gamma + \frac12 \le (2r)^{-1} + \frac12 \le 1$)
  \begin{align*}
    &\pe \intnstom \ure^{\gamma r s} |(\gamma - \tfrac12) v(x, s) + u (x, s)|^r \\
    &\le  2^{r-1} \intnstom \ure^{\gamma rs} (v(x, s))^r + 2^{r-1} \intnstom \ure^{\gamma rs} (u(x, s))^r \\
    &\le 3 \cdot 2^{r-1} \intnstom \ure^{\gamma rs} (u(x, s))^r + 2^{r-1} \int_0^1 v_0^r
  \end{align*}
  for all $t \in (0, \tmax]$. That is, we obtain \eqref{eq:max_reg:est} for a certain $C > 0$.
\end{proof}

We are now ready to show the key part of the proof of Proposition~\ref{prop:eps_reg},
namely that $\varepsilon$-regularity with respect to $L^1$ (i.e., \eqref{eq:eps_reg:eps_cond}) implies an $L^p$ estimate for arbitrarily large $p>1$.
\begin{lemma}\label{lm:eqs_reg_lp}
  Let $(\tau, k) \in \{(0, 0), (0, 1), (1, 1)\}$, $m \in \R$ and $p \in (\max\{1, 1-m\}, \infty)$.
  Then there exists $\eps(p) > 0$ with the following property:
  If there are $\delta > 0$ and initial data $u_0, v_0$ fulfilling \eqref{initial_data} such that
  the solution $(u, v)$ of \eqref{system} given by Lemma~\ref{lm:local_ex} fulfils \eqref{eq:eps_reg:eps_cond} (with $\eps$ replaced by $\eps(p)$),
  then there is $C > 0$ such that
  \begin{align}\label{eq:eps_reg_lp:statement}
      \intom u^p(\cdot, t) \le C
      \qquad \text{for all $t \in (0, \tmax)$}.
  \end{align}
\end{lemma}
\begin{proof}
  Recalling the regularity and uniqueness properties asserted in Lemma~\ref{lm:local_ex} and switching to the solution emanating from $(u(\cdot, t_0), v(\cdot, t_0))$ for some $t_0 \in (0, \tmax)$, if necessary and if $\tau = 1$, we may without loss of generality assume $\tau v_0 \in \con2$ and $\partial_\nu (\tau v_0) = 0$ on $\partial \Omega$.

  We set $c_1 \defs \frac{2p(p-1)}{(p+m-1)^2} > 0$ and $c_2 \defs \frac{p(p-1)}{p+m} > 0$,
  let $c_3 > 0$ be as given by Lemma~\ref{lm:max_reg} applied to $r \defs p+m+1 \in (1, \infty)$
  and $\eps > 0$ be as given by Lemma~\ref{lm:eps_inq} applied to $\alpha \defs p+m+1 > 2$ and $\eta \defs \frac{c_1}{c_2+c_2c_3} > 0$,
  and then set $\eps(p) \defs \frac\eps2$.

  Supposing now that \eqref{eq:eps_reg:eps_cond} (with $\eps$ replaced by $\eps(p)$) is fulfilled for some $\delta > 0$,
  using that constants are uniformly integrable on $(0, 1)$ and shrinking $\delta$, if necessary, we see that $\varphi = u+1$ satisfies \eqref{eq:esp_inq}.
  Thus, by Lemma~\ref{lm:eps_inq} there exists $c_4>0$ such that
  \begin{align}\label{eq:eps_reg_lp:eps_gni}
    (c_2+c_2c_3) \intom (u+1)^{p+m+1}
    \le c_1 \intom \big((u+1)^{\frac{p+m-1}{2}}\big)_x^2 + c_4
    \qquad \text{in $(0, \tmax)$}.
  \end{align}

  Testing the first equation in \eqref{system} with $(u+1)^{p-1}$ yields
  \begin{align}\label{eq:eps_reg_lp:ddt_p}
    &\pe  \ddt \intom (u+1)^p + p(p-1) \intom (u+1)^{p+m-3} u_x^2 \notag \\
    &=    p(p-1) \intom u (u+1)^{p+m-2} u_x v_x \notag \\
    &=    p(p-1) \intom  \left(\int_0^u s (s+1)^{p+m-2} \ds \right)_x v_x \notag \\
    &\le  \frac{p(p-1)}{p+m} \intom (u+1)^{p+m} |v_{xx}| \notag \\
    &\le  \frac{p(p-1)}{p+m} \left( \intom (u+1)^{p+m+1} + \intom |v_{xx}|^{p+m+1} \right)
    \qquad \text{in $(0, \tmax)$}.
  \end{align}
  Here, the dissipative term on the left-hand side can be rewritten as
  \begin{align}\label{eq:eps_reg_lp:diss}
          p(p-1) \intom (u+1)^{p+m-3} u_x^2
    &=    2c_1 \intom \big((u+1)^{\frac{p+m-1}{2}}\big)_x^2
    \qquad \text{in $(0, \tmax)$},
  \end{align}
  where we recall that $c_1 = \frac{2p(p-1)}{(p+m-1)^2}$.
  By the Gagliardo--Nirenberg inequality, there are $c_5 > 0$ and $\tilde \lambda > 0$ such that
  \begin{align*}
        \left( \int_0^1 |\varphi|^\frac{2p}{p+m-1} \right)^{\tilde \lambda}
    \le c_5 \left( \int_0^1 |\varphi_x|^2 \right) \left( \int_0^1 |\varphi|^\frac{2}{p+m-1} \right)^{p \tilde \lambda - (p+m-1)}
        + c_5 \left( \int_0^1 |\varphi|^\frac{2}{p+m-1} \right)^{p \tilde \lambda}
  \end{align*}
  for all $\varphi \in \sob12$,
  hence
  \begin{align}\label{eq:eps_reg_lp:diss_up}
          c_1 \int_0^1\big((u+1)^{\frac{p+m-1}{2}}\big)_x^2
    &\ge  c_6 \left( \int_0^1 (u+1)^p \right)^{\tilde \lambda} - c_7
     \ge  c_8 \left( \int_0^1 (u+1)^p \right)^\lambda - c_7
    \qquad \text{in $(0, \tmax)$}
  \end{align}
  for certain $c_6, c_7 > 0$, $\lambda \defs \min\{\tilde \lambda, \frac12\} \in (0, 1)$ and $c_8 \defs \min\{c_6, \frac1{4r}\} \in (0, \frac{1}{2r})$.

  Next, we fix $T \in (0, \tmax)$ but emphasize that all constants $c_i$ below do not depend on $T$.
  Then
  \begin{align*}
  A \defs \sup_{t \in (0, T)} \intom (u(\cdot, t)+1)^p \ge 1
  \end{align*} is finite
  and combining \eqref{eq:eps_reg_lp:ddt_p}, \eqref{eq:eps_reg_lp:diss} and \eqref{eq:eps_reg_lp:diss_up} yields
  \begin{align*}
    &\pe  \ddt \intom (u+1)^p\\
    &\le  -A^{\lambda-1} c_8 \intom (u+1)^p - c_1 \intom \big((u+1)^{\frac{p+m-1}{2}}\big)_x^2 + c_7 + c_2 \intom (u+1)^{p+m+1} + c_2 \intom |v_{xx}|^{p+m+1}
  \end{align*}
  in $(0, T)$ since  $c_2 = \frac{p(p-1)}{p+m}$.
  Thanks to the variation-of-constants formula,
  \begin{align*}
          \intom (u(\cdot, t)+1)^p
    &\le  \ure^{-A^{\lambda-1} c_8 t} \intom (u_0+1)^p
          - c_1 \intnstom \ure^{-A^{\lambda-1} c_8 (t-s)} \big((u(x, s)+1)^{\frac{p+m-1}{2}}\big)_x^2 \dx \ds \notag \\
    &\pe  + c_7 \int_0^t \ure^{-A^{\lambda-1} c_8 (t-s)} \ds
          + c_2 \intnstom \ure^{-A^{\lambda-1} c_8 (t-s)} (u(x, s)+1)^{p+m+1} \dx \ds \notag \\
    &\pe  + c_2 \intnstom \ure^{-A^{\lambda-1} c_8 (t-s)} |v_{xx}(x, s)|^{p+m+1} \dx \ds \notag \\
    &\sfed I_1(t) + I_2(t) + I_3(t) + I_4(t) + I_5(t)
    \qquad \text{for all $t \in (0, T)$}.
  \end{align*}
  Recalling that $c_3 > 0$ is the constant given by Lemma~\ref{lm:max_reg} for $r = p+m+1$
  and noting that  $\gamma \defs A^{\lambda-1} c_8 \le \frac1{2r}$,
  we may conclude that
  \begin{align*}
    &\pe  \intnstom \ure^{-A^{\lambda-1} c_8 (t-s)} |v_{xx}(x, s)|^{p+m+1} \dx \ds \\
    &\le  c_3 \intnstom \ure^{-A^{\lambda-1} c_8 (t-s)} (u(x, s))^{p+m+1} \dx \ds
          + c_3 \|\tau v_0\|_{\sob2{p+m+1}}^{p+m+1} \\
    &\le  c_3 \intnstom \ure^{-A^{\lambda-1} c_8 (t-s)} (u(x, s)+1)^{p+m+1} \dx \ds
          + c_9
  \end{align*}
  for all $t \in (0, T)$, where $c_9 \defs c_3 \|\tau v_0\|_{\sob2{p+m+1}}^{p+m+1} \ge 0$.

  The integral on the right-hand side herein can be estimated by using \eqref{eq:eps_reg_lp:eps_gni};
  for all $t \in (0, T)$, we have
  \begin{align*}\nn
    &\pe (c_2+c_2c_3) \intnstom \ure^{-A^{\lambda-1} c_8 (t-s)} (u(x, s)+1)^{p+m+1} \dx \ds \\
    & \le c_1\intnstom \ure^{-A^{\lambda-1} c_8 (t-s)} \big((u(x, s)+1)^{\frac{p+m-1}{2}}\big)_x^2 \dx \ds + c_4\int_0^t \ure^{-A^{\lambda-1} c_8 (t-s)} \ds,
  \end{align*}
  so that
  \begin{align*}
     I_2(t) + I_4(t) + I_5(t)
     \le c_9 + \frac{c_4}{c_7} I_3(t)
     \qquad \text{for all $t \in (0, T)$}.
  \end{align*}
  Since a direct integration moreover gives
  \begin{align*}
          \int_0^t \ure^{-A^{\lambda-1} c_8 (t-s)} \ds
    \le   \frac{A^{1-\lambda}}{c_8}
    \qquad \text{for all $t \in (0, T)$},
  \end{align*}
  and as
  \begin{align*}
      I_1(t) \le I_1(0) \le A^{1-\lambda} \intom (u_0+1)^p \sfed A^{1-\lambda} c_{10}
      \qquad \text{for all $t \in (0, T)$},
  \end{align*}
  we conclude
  \begin{align*}
      \int_0^1 (u(\cdot, t) + 1)^p
      \le A^{1-\lambda} \left(c_{10} + \frac{c_7}{c_8} + \frac{c_4}{c_8}\right) + c_9
      \le \frac12 A + c_{11}
      \qquad \text{for all $t \in (0, T)$}
  \end{align*}
  for some $c_{11} > 0$, where in the last step we have used Young's inequality and that $\lambda \in (0, 1)$.
  Taking the supremum over $(0, T)$ yields $A \le 2c_{11}$,
  so that the statement follows upon taking $T \nea \tmax$.
\end{proof}

Together with a well-established Moser iteration argument, we are ready to prove the boundedness criterion in Proposition~\ref{prop:eps_reg}.
\begin{proof}[Proof of Proposition~\ref{prop:eps_reg}]
  Local existence of nonnegative, uniquely determined solutions in the regularity class \eqref{eq:eps_reg:reg} has been asserted in Lemma~\ref{lm:local_ex}.

  By \cite[Lemma~A.1]{TaoWinklerBoundednessQuasilinearParabolic2012}, there is $p_0 \in (1, \infty)$ such that if $\int_0^1 u^p$ is bounded in $(0, \tmax)$ for any $p \ge p_0$, then so is $t \mapsto \|u(\cdot, t)\|_{\leb \infty}$.
  Thus, we fix $p \defs \max\{p_0, 2-m\} > \max\{1, 1-m\}$ and set $\eps \defs \eps(p)$ with $\eps(p)$ given by Lemma~\ref{lm:eqs_reg_lp}.
  If \eqref{eq:eps_reg:eps_cond} holds for this $\eps$ and some $\delta > 0$, then Lemma~\ref{lm:eqs_reg_lp} yields \eqref{eq:eps_reg_lp:statement} for some $C > 0$, whence $u$ is bounded in $(0, 1) \times (0, \tmax)$.
\end{proof}

We close this section by discussing the condition \eqref{eq:eps_reg:eps_cond}.

\begin{prop}\label{prop:eps_reg_alt_cond}
  Each of the following conditions implies \eqref{eq:eps_reg:eps_cond} (even for all $\eps > 0$ and certain $\delta > 0$ depending on $\eps$).
  \begin{enumerate}
    \item[(i)]
      The family $\{\, u(\cdot, t) \mid t \in (0, \tmax)\,\}$ is uniformly integrable on $(0, 1)$.

    \item[(ii)]
      There exists $\Phi \in C^0([0, \infty); [0, \infty))$ with $\frac{\Phi(\xi)}{\xi} \to \infty$ as $\xi \to \infty$ such that $\sup_{t \in (0, \tmax)} \int_0^1 \Phi(u(\cdot, t)) < \infty$.

    \item[(iii)]
      The function $w \colon [0, 1] \times [0, \tmax) \to [0, \infty)$, $w(x, t) = \int_0^x u(\rho, t) \drho$, is uniformly continuous in space uniformly in time, i.e.,
      \begin{align}\label{eq:eps_reg_alt_cond:uniform_cont}
        \forall \eps > 0\;
        \exists \delta > 0\;
        \forall t \in (0, \tmax)\;
        \forall x_1, x_2 \in [0, 1] \text{ with } |x_1 - x_2| < \delta:
        |w(x_1, t) - w(x_2, t)| < \eps.
       \end{align}
  \end{enumerate}
\end{prop}
\begin{proof}
  Condition~(i) requires a similar condition for intervals replaced by all measurable sets $A \subset (0, 1)$ with $|A| < \delta$ and hence directly implies \eqref{eq:eps_reg:eps_cond}.
  The de la Vall\'ee Poussin theorem asserts equivalence of (i) and (ii).
  Since $\int_a^b u(\cdot, t) = w(b, t) - w(a, t)$ for all $0 \le a \le b \le 1$ and all $t \in (0, \tmax)$, also condition~(iii) implies \eqref{eq:eps_reg:eps_cond} for all $\eps > 0$.
\end{proof}

However, as the following elementary example shows, even uniform-in-time Hölder continuity of the function $w$ defined in Proposition~\ref{prop:eps_reg_alt_cond}(iii) does not imply Proposition~\ref{prop:eps_reg_alt_cond}(i) for general functions $u$.
Accordingly, for the proof of Theorem~\ref{th:gb_pe} in Section~\ref{sec:gb_pe} below,
it is crucial that we have obtained the conclusion of Proposition~\ref{prop:eps_reg} not only when Proposition~\ref{prop:eps_reg_alt_cond}(i) but also when just \eqref{eq:eps_reg:eps_cond} holds.
\begin{prop}\label{lm:ex_w_hoelder_u_not_equiint}
  \begin{enumerate}
    \item[(i)]
    There exists $(w_k)_{k \in \N} \subseteq X \defs C^\frac12([0, 1]) \cap W^{1, 1}((0, 1))$ with
    \begin{align}\label{eq:ex_w_hoelder_u_not_equiint:hoelder}
      \|w_k\|_{C^\frac12([0, 1])} \le 1+\sqrt3, \quad
      u_k \defs w_{k x} \ge 0 \text{ a.e.} \quad \text{and} \quad
      \int_0^1 u_k = 1
    \end{align}
    being such that
    \begin{align}\label{eq:ex_w_hoelder_u_not_equiint:not_equi}
      \{\, u_k \mid k \in \N\,\} \text{ is not uniformly integrable on $(0, 1)$}.
    \end{align}

    \item[(ii)]
    There exists $w \in C^0([0, 1); X) \cap L^\infty((0, 1); X)$ satisfying \eqref{eq:eps_reg_alt_cond:uniform_cont} with $\tmax = 1$, but which is such that $\{\, w_x(\cdot, t) \mid t \in (0, 1)\,\}$ is not uniformly integrable on $(0, 1)$.
  \end{enumerate}
\end{prop}
Since we are not aware of any concrete example in the literature, we present a short proof in Appendix~\ref{appendix} below.

\section{Global boundedness by Hölder regularity: proof of Theorem~\ref{th:gb_pe}}\label{sec:gb_pe}
In line with the strategy outlined above, we now aim to prove boundedness of solutions $(u, v)$ to \eqref{system} by verifying Proposition~\ref{prop:eps_reg_alt_cond}(iii), that is, that the mass accumulation function
\begin{align}\label{eq:mass_acc:def_w}
    w \colon [0, 1] \times [0, \tmax) \to [0, \infty), \quad w(x, t) = \int_0^x u(\rho, t) \drho,
\end{align}
whose higher dimensional versions have been introduced for the study of chemotaxis systems in \cite{JagerLuckhausExplosionsSolutionsSystem1992},
is uniformly continuous in $x$ uniformly in $t$.
To that end, we first recall some of its basic properties, which are yet independent of the form of the second equation in \eqref{system}.
\begin{lemma}\label{lm:mass_acc}
  Let $(\tau, k) \in \{(0, 0), (0, 1), (1, 1)\}$, let $m \in \R$, suppose \eqref{initial_data}
  and denote the solution of \eqref{system} given by Proposition~\ref{prop:eps_reg} by $(u, v)$.
  Then the function $w$ defined in \eqref{eq:mass_acc:def_w}
  belongs to $C^0([0, \tmax); C^1([0, 1])) \cap C^{\infty}((0, 1) \times (0, \tmax))$ and solves
  \begin{align}\label{eq:mass:acc:w_sol}
    \begin{cases}
      w_t = (w_x + 1)^{m-1} w_{xx} - w_x (w_x+1)^m v_x & \text{in $(0, 1) \times (0, \tmax)$}, \\
      w(0, \cdot) = 0,\; w(1, \cdot) = M               & \text{in $(0, \tmax)$}, \\
      w(\cdot, 0) = w_0 \defs \int_0^{{\text{\tiny $\bullet$}}} u_0 & \text{in $(0, 1)$}.
    \end{cases}
  \end{align}
\end{lemma}
\begin{proof}
  The claimed regularity follows from \eqref{eq:eps_reg:reg}.
  Inserting the first equation in \eqref{system} gives
  \begin{align*}
        w_t
    &=  \int_0^x u_t(\rho, t) \drho
    =   \int_0^x \big((u+1)^{m-1} u_x - u(u+1)^m v_x\big)_x \drho \\
    &=  (u+1)^{m-1} u_x - u(u+1)^m v_x
    =   (w_x + 1)^{m-1} w_{xx} - w_x (w_x+1)^m v_x
  \end{align*}
  in $(0, 1) \times (0, \tmax)$.
  Moreover, the statement $w(1, t) = M$ for all $t \in (0, \tmax)$ is equivalent to \eqref{eq:u_mass},
  and the remaining identities on the parabolic boundary directly follow from the definition of $w$.
\end{proof}

Multi-dimensional counterparts of \eqref{eq:mass:acc:w_sol} have been used in the analysis of several chemotaxis systems, both to show blow-up (see, for instance, \cite{JagerLuckhausExplosionsSolutionsSystem1992} for the first such application and the survey \cite{LankeitWinklerFacingLowRegularity2019} for more examples) and, by applying comparison theorems to variants of \eqref{eq:mass:acc:w_sol}, also boundedness of $u$ (cf., e.g., \cite{WinklerHowUnstableSpatial2018}, \cite{MaoLiNote8pProblem2024} and Section~\ref{sec:gb_jl} below).

Here, we follow a different approach, namely, that well-known Hölder regularity theory (cf.\ \cite{DiBenedettoDegenerateParabolicEquations1993}) becomes applicable if the inhomogeneity $v_x$ in \eqref{eq:mass:acc:w_sol} is bounded.
In the fully parabolic setting, semigroup estimates and \eqref{eq:u_mass} show boundedness of $v_x$ in $L^\infty((0, \tmax); L^p((0, 1)))$ for all $p \in [1, \infty)$ -- which turns out to be barely insufficient for the subsequent analysis in Lemma~\ref{lm:w_hoelder} below.
If $\tau = 0$, however, the desired bound can be proven in a straightforward manner.
This is stated in the following lemma, which already also addresses a (weaker) bound for $\tau = 1$
to be used in the proof of Lemma~\ref{lm:ddt_psi_u} in Section~\ref{sec:gb_pp}.
\begin{lemma}\label{lm:vx_linfty}
  Let $(\tau,k)\in \{(0,0),(0,1),(1,1)\}$  and $m \in \R$,
  and suppose \eqref{initial_data}.
  Then the solution $(u, v)$ of \eqref{system} given by Proposition~\ref{prop:eps_reg} fulfils:
  \begin{itemize}
    \item[(i)] If $(\tau,k)\in\{(0,0),(0,1)\}$, then $|v_x| \le 2M = 2 \int_0^1 u_0$ in $(0, 1) \times (0, \tmax)$.
    \item[(ii)] There exists $C>0$ such that $|v_x|^2 \le C \int_0^1 (\tau v_t)^2 + C$ in $(0, 1) \times (0, \tmax)$.
  \end{itemize}
\end{lemma}
\begin{proof}
  Integrating the second equation in \eqref{system} and recalling \eqref{eq:u_mass} gives
  \begin{align*}
        |v_x(x, t)|
    &=   \left| \int_0^x \big( \tau v_t(\rho,t) - u(\rho, t) + kv(\rho, t) + (1-k)M \big) \drho \right|\\
    &\le\int_0^1| \tau  v_t|+\int_0^1 \big( u + kv + (1-k)M \big)
    \le \int_0^1 |\tau v_t|+2M +k\norm[L^1((0,1))]{\tau v_0}
  \end{align*}
  for all $(x, t) \in (0, 1) \times (0, \tmax)$, which implies (i).
  Moreover, we can continue to estimate
  \begin{align*}
        |v_x|^2
    &=  \left(
    \int_0^1 |\tau v_t|+2M +k\norm[L^1((0,1))]{\tau v_0}
    \right)^2 \\
    &\le   3 \int_0^1 (\tau v_t)^2 + 12M^2+3k^2\norm[L^1((0,1))]{\tau v_0}^2
    \qquad \text{ in $(0, 1) \times (0, \tmax)$}.
  \end{align*}
  Therefore, (ii) follows upon an evident choice of $C$.
\end{proof}

With this preparation at hand, we are able to apply Hölder regularity theory to the problem \eqref{eq:mass:acc:w_sol} and subsequently prove Theorem~\ref{th:gb_pe}.
\begin{lemma}\label{lm:w_hoelder}
  Let $(\tau,k)\in \{ (0,0),(0, 1)\}$ and $m > 0$,
  suppose \eqref{initial_data} and let $(u, v)$ be the solution of \eqref{system} given by Proposition~\ref{prop:eps_reg}.
  Then there are $\alpha \in (0, 1)$ and $C > 0$ such that
  \begin{align}\label{eq:w_hoelder:statement}
    \|w(\cdot, t)\|_{C^\alpha([0, 1])} \le C
    \qquad \text{for all $t \in [\min\{1, \tfrac{\tmax}{2}\}, \tmax)$}.
  \end{align}
\end{lemma}
\begin{proof}
  Aiming to apply the Hölder regularity results from  \cite{DiBenedettoDegenerateParabolicEquations1993},
  we first note that due to $w_x \ge 0$ the PDE in \eqref{eq:mass:acc:w_sol} can be rewritten as
  \begin{align}\label{eq:w_hoelder:w_eq}
    w_t = \big(a(x, t, w, w_x)\big)_x + b(x, t, w, w_x)
    \qquad \text{in $(0, 1) \times (0, \tmax)$},
  \end{align}
  where
  \[
    a(x, t, z, \xi) \defs \frac1m (|\xi|+1)^{m-1} (\xi+1)
    \quad \text{and} \quad
    b(x, t, z, \xi) \defs -|\xi|(|\xi|+1)^m v_x(x, t)
  \]
  for $(x, t, z, \xi) \in [0, 1] \times [0, \tmax) \times \R \times \R$.

  Next, Young's inequality, the elementary estimate $(s_1 + s_2)^p \le 2^{(p-1)_+}(s_1^p + s_2^p)$ holding for all $s_1, s_2, p \ge 0$, and Lemma~\ref{lm:vx_linfty} allow us to estimate
  \begin{align*}
          a(x, t, z, \xi) \xi
    &\ge  \frac1m (|\xi|+1)^{m+1} - \frac1m (|\xi|+1)^m
    \ge   \frac1{2m} (|\xi|+1)^{m+1} - c_1
    \ge   \frac1{2m} |\xi|^{m+1} - c_1 \\
          |a(x, t, z, \xi)|
    &\le  c_2 |\xi|^m + c_2 \\
          |b(x, t, z, \xi)|
    &\le  2M (|\xi|+1)^{m+1}
    \le   c_3 |\xi|^{m+1} + c_3
  \end{align*}
  for all $(x, t, z, \xi) \in [0, 1] \times [0, \tmax) \times \R \times \R$ and certain $c_1, c_2, c_3$ only depending on $m$ and $M$.
  Thus, the structure conditions (A$_1$)--(A$_3$) in \cite[p.~16]{DiBenedettoDegenerateParabolicEquations1993} (and \cite[p.~148]{PorzioVespriHolderEstimatesLocal1993}) are fulfilled for $p \defs m+1 > 1$ and constant $\varphi_i$.
  The latter then trivially fulfil (A$_4)$, (A$_5$) and (A$_5$-i) in \cite[p.~17]{DiBenedettoDegenerateParabolicEquations1993} (and the corresponding conditions in \cite{PorzioVespriHolderEstimatesLocal1993}).

  By \eqref{eq:w_hoelder:w_eq} and as $w$ is constant on the spatial boundary,  \cite[Chapter~II and Chapter~III]{DiBenedettoDegenerateParabolicEquations1993} (respectively, \cite{PorzioVespriHolderEstimatesLocal1993} for the non-degenerate, non-singular case $m = 1$) yields $C > 0$ and $\alpha \in (0, 1)$ such that setting $T \defs  \min\{1, \frac{\tmax}{2}\}$, we have
  \begin{align*}
    \|w\|_{C^{\alpha, \alpha/2}([0, 1] \times [t_0, t_0 + T])} \le C
    \qquad \text{for all $t_0 \in [T, \tmax - T]$},
  \end{align*}
  which entails \eqref{eq:w_hoelder:statement}.
\end{proof}

\begin{proof}[Proof of Theorem~\ref{th:gb_pe}]
  Lemma~\ref{lm:w_hoelder} combined with the inclusion $w \in C^0([0, 1] \times [0, \min\{1, \tfrac{\tmax}{2}\}])$ shows that \eqref{eq:eps_reg_alt_cond:uniform_cont} is fulfilled. Hence, Proposition~\ref{prop:eps_reg_alt_cond}(iii) and Proposition~\ref{prop:eps_reg} yield global boundedness of $(u, v)$.
\end{proof}

\section{Global boundedness by comparison and monotonicity: proof of Theorem~\ref{th:gb_jl}}\label{sec:gb_jl}
In the parabolic--elliptic setting, but without any restriction on $m$, the problem \eqref{eq:mass:acc:w_sol} admits a constant-in-time Hölder continuous supersolution vanishing at the spatial origin. By means of a comparison theorem, we thus obtain the following.
\begin{lemma}\label{lm:w_le_x_alpha}
  Let $(\tau ,k)\in \{(0,0), (0, 1)\}$ and $m \in \R$, suppose \eqref{initial_data} and let $(u, v)$ be the solution of \eqref{system} given by Proposition~\ref{prop:eps_reg}.
  Then there are $C > 0$ and $\alpha \in (0, 1)$ such that the function $w$ defined in \eqref{eq:mass_acc:def_w} fulfils
  \begin{align*}
    w(x, t) \le C x^\alpha
    \qquad \text{for all $x \in (0, 1)$ and all $t \in (0, \tmax)$}.
  \end{align*}
\end{lemma}
\begin{proof}
  We set
  \begin{align}\label{eq:w_le_x_alpha:x0_c0}
    x_0 \defs \frac{1}{8\max\{M, 1\}} \in (0, 1)
    \quad \text{and} \quad
    c_1 \defs \sup_{x \in (0, x_0)} x^{-\frac12} w_0(x)
  \end{align}
  as well as
  \begin{align}\label{eq:w_le_x_alpha:c2_alpha}
    \alpha \defs \frac{1}{8\max\{M^2, Mc_1, 1\}} \in \left(0, \frac12\right)
    \quad \text{and} \quad
    c_2 \defs \max\{x_0^{-\alpha} M, c_1\}.
  \end{align}
  Then
  \begin{align*}
    \ol w \colon [0, x_0] \times [0, \tmax) \to \R, \quad
    \ol w(x, t) = c_2 x^\alpha,
  \end{align*}
  fulfils
  \begin{align}\label{eq:w_le_x_alpha:parab_bdry}
    \ol w(0, t)   = 0 = w(0, t), \quad
    \ol w(x_0, t) = c_2 x_0^\alpha \ge M \ge w(x_0, t) \quad \text{and} \quad
    \ol w(x, 0)   = c_2 x^\alpha \ge c_1 x^\frac12 \ge w_0(x)
  \end{align}
  for all $(x, t) \in [0, x_0] \times [0, \tmax)$.

  Next, for $\varphi \in C^{2, 1}((0, x_0) \times (0, \tmax))$, we set
  \begin{align*}
    \mc P \varphi \defs \varphi_t - (\varphi_x + 1)^{m-1} \varphi_{xx} - 2M \varphi_x (\varphi_x + 1)^m.
  \end{align*}
  We compute $\ol w_x(x, t) = c_2 \alpha x^{\alpha-1} > 0$ and $\ol w_{xx}(x, t) = -c_2\alpha(1-\alpha) x^{\alpha-2}$ for all $(x, t) \in (0, 1) \times (0, \tmax)$,
  so that
  \begin{align*}
          \mc P \ol w
    &=    0 + (\ol w_x + 1)^{m-1} \left( -\ol w_{xx} - 2M \ol w_x (\ol w_x + 1) \right) \\
    &\ge (\ol w_x + 1)^{m-1} c_2 \alpha x^{\alpha-2} \left( \frac12 - 2M c_2 \alpha x^\alpha - 2M x \right)
  \end{align*}
  in $(0, 1) \times (0, \tmax)$.
  Recalling \eqref{eq:w_le_x_alpha:x0_c0} and \eqref{eq:w_le_x_alpha:c2_alpha}, we have herein
  \begin{align*}
          2M c_2 \alpha x^\alpha
    &\le  2M \alpha \max\{M, c_1\}
    \le   \frac14
    \quad \text{and} \quad
          2M x
    \le   2M x_0
    \le   \frac14
  \end{align*}
  for all $x \in (0, x_0)$ and hence $\mc P\ol w \ge 0$ in $(0, 1) \times (0, \tmax)$.

  On the other hand, \eqref{eq:mass:acc:w_sol} and Lemma~\ref{lm:vx_linfty}(i) assert $\mc Pw \le 0$ in $(0, x_0) \times (0, \tmax)$.
  As moreover $w \le \ol w$ on the parabolic boundary by \eqref{eq:w_le_x_alpha:parab_bdry},
  the comparison principle \cite[Theorem A.1]{bellomo2017finite} yields $w \le \ol w$ in $(0, x_0) \times (0, \tmax)$.
  Since $w \le M x^\alpha x^{-\alpha} \le c_2 x^\alpha$ in $[x_0, 1) \times (0, \tmax)$, the statement follows upon setting $C \defs c_2$.
\end{proof}

To prepare proving regularity properties of $w$ away from the spatial origin, we next note that monotonicity of $u_0$ propagates in the J\"ager--Luckhaus setting.
\begin{lemma}\label{lm:u_monotone}
  Let $(\tau,k) = (0,0)$ and $m \in \R$.
  Suppose \eqref{initial_data} and that $u_0$ is nonincreasing.
  Then the first component of the solution $(u, v)$ of \eqref{system} given by Proposition~\ref{prop:eps_reg} is nonincreasing in space throughout $(0, \tmax)$.
\end{lemma}
\begin{proof}
  The idea of proof follows from \cite[Lemma~2.2]{WinklerCriticalBlowupExponent2018}, some adaptations for quasilinear systems have been made in \cite[Lemma~5.1]{BlackEtAlRelaxedParameterConditions2021}.
  Since $v_{xx} = M-u$ by the second equation in \eqref{system}, the first equation in \eqref{system} is equivalent to $u_t = (D(u) u_x)_x - S(u)(M-u) - S'(u) u_x v_x$,
  which upon differentiating with respect to $x$ and recalling the assumptions on the parabolic boundary results in
  \begin{align*}
    \begin{cases}
      (u_x)_t = D(u)(u_x)_{xx} + a(x,t) (u_x)_x+ b(x,t) u_x & \text{in $(0, 1) \times (0, \tmax)$}, \\
      u_x = 0                                               & \text{on $\{0, 1\} \times (0, \tmax)$}, \\
      u_{0x}(\cdot,0) \le 0                                 & \text{in $(0, 1)$},
    \end{cases}
  \end{align*}
  where
  \begin{align*}
            a
    &\defs  2D'(u)u_x
            - S'(u)v_x
    \quad \text{and} \quad
            b
     \defs  D''(u) u_x^2
            - 2S'(u)(M-u)
            + S(u)
            - S''(u) u_x v_x
  \end{align*}
  in $(0,1)\times(0,\tmax)$.
  The claim then follows by the usual comparison principle argument:
  For $T \in (0, \tmax)$ and $\eps > 0$, we consider $z(x,t)=u_x(x,t)-\eps \ure^{\lambda t}$ for $x \in [0, 1]$, $t \in [0, T]$ and $\lambda = 2 \|b\|_{C^0([0, 1] \times [0, T])}$,
  and assume for the sake of contradiction that $z$ is not nonpositive everywhere.
  Since $z$ is negative on the parabolic boundary, this would imply that there are $x_0 \in (0, 1)$ and $t_0 \in (0, T)$ with $z(x_0, t_0) = 0$ but $z \le 0$ in $[0, 1] \times [0, t_0]$.
  At this point, $z_t \ge 0$, $z_{xx} \le 0$ and $z_x = 0$, whence
  \begin{align*}
    0 \le D(u(x_0, t_0)) \cdot 0 + a(x_0, t_0) \cdot 0 + b(x_0, t_0) \cdot(0 + \eps \ure^{\lambda t_0}) - \lambda \eps \ure^{\lambda t_0} < 0,
  \end{align*}
  a contradiction.
\end{proof}

The supersolution from Lemma~\ref{lm:w_le_x_alpha} at first only provides bounds for $w$, not for its derivative. However, together with persistence of the concavity of $w$ guaranteed by Lemma~\ref{lm:u_monotone}, this implies an $L^p$ estimate for $w_x = u$ for some $p>1$.
\begin{lemma}\label{lm:u_lp_jl}
  Let $(\tau,k )=(0,0)$ and $m \in \R$.
  Suppose \eqref{initial_data} and that $u_0$ is nonincreasing.
  Then there are $p > 1$ and $C > 0$ such that the solution $(u, v)$ of \eqref{system} given by Proposition~\ref{prop:eps_reg} fulfils
  \begin{align*}
    \intom u^p(\cdot, t) \le C
    \qquad \text{for all $t \in (0, \tmax)$}.
  \end{align*}
\end{lemma}
\begin{proof}
  Similarly as in \cite[Lemma~3.3]{FuestApproachingOptimalityBlowup2021}, we infer from the mean value theorem and Lemma~\ref{lm:u_monotone} that
  \begin{align*}
      \frac{w(x, t)}{x}
    = \frac{w(x, t) - w(0, t)}{x - 0}
    = w_x(\xi(t), t)
    = u(\xi(t), t)
    \ge u(x, t)
  \end{align*}
  for certain $\xi(t) \in (0, x)$ and all $(x, t) \in (0, 1) \times (0, \tmax)$.
  Thus, as $x \mapsto x^{\alpha-1} \in \leb p$ for all $p \in (1, \frac{1}{1-\alpha})$ and all $\alpha \in (0, 1)$,
  the desired estimate follows from Lemma~\ref{lm:w_le_x_alpha}.
\end{proof}

Boundedness of $u$ then can be readily achieved by the refined criterion in Proposition~\ref{prop:eps_reg}.

\begin{proof}[Proof of Theorem~\ref{th:gb_jl}]
  According to Lemma~\ref{lm:u_lp_jl} and Proposition~\ref{prop:eps_reg_alt_cond}(ii), we may conclude global boundedness of the solution $(u, v)$ of \eqref{system} from Proposition~\ref{prop:eps_reg}.
\end{proof}

\section{Global boundedness due to bounded sensitivity and energy: proof of Theorem~\ref{th:gb_pp}}\label{sec:gb_pp}
Our analysis in this section makes use of the Lyapunov functional
\begin{align}\label{eq:energy:def_F}
          \mathcal F(\tilde u, \tilde v)
  &\defs  \int_0^1 G(\tilde u)
          - \int_0^1 \tilde u \tilde v
          + \frac12 \int_0^1 \tilde v_x^2
          + \frac k2 \int_0^1 \tilde v^2, \\
  \label{eq:energy:def_G}
          G(s)
  &\defs  \int_1^s\int_1^\sigma \frac{1}{\rho(\rho+1)} \drho \dsigma,
\end{align}
defined for $(\tilde u, \tilde v) \in C^0([0, 1]; (0, \infty)) \times \con1$ and $s \in (0, \infty)$.
Related functionals have been widely used for the analysis of Keller--Segel systems,
see for instance \cite{NagaiEtAlApplicationTrudingerMoser1997} and \cite{GajewskiZachariasGlobalBehaviourReactiondiffusion1998} for early boundedness proofs
and the survey \cite{LankeitWinklerFacingLowRegularity2019} for an overview of blow-up techniques relying on the Lyapunov functional.

We first note that $\mc F$, indeed, decreases along trajectories of \eqref{system}.
\begin{lemma}\label{lm:energy}
  Let $(\tau,k)\in \{(0,0),(0,1),(1,1)\}$ and $m \in \R$, and suppose (\ref{initial_data}).
  Then the solution $(u, v)$ of \eqref{system} given by Proposition~\ref{prop:eps_reg} fulfils
  \begin{align}\label{eq:energy:est}
    \mathcal F(u(\cdot,t),v(\cdot,t)) + \int_{t_0}^{t} \intom (\tau v_t)^2 \le \mathcal F(u(\cdot, t_0), v(\cdot, t_0)) \qquad \text{for all } 0 < t_0 < t < \tmax.
  \end{align}
\end{lemma}
\begin{proof}
  Since $u$ and $v$ are smooth and positive on compact subsets of $[0, 1] \times (0, \tmax)$,
  we obtain
  \begin{align*}
        \ddt \mc F(u, v)
    &=  - \int_0^1 \big((u+1)^{m-1} u_x - u(u+1)^m v_x\big) \big( G'(u) - v \big)_x
        - \int_0^1 v_t \big(u + v_{xx} - k v - (1-k) M\big) \\
    &=  - \int_0^1 u (u+1)^m \left|\frac{1}{u(u+1)} u_x - v_x\right|^2
        - \int_0^1 (\tau v_t)^2
     \le  - \int_0^1 (\tau v_t)^2
    \qquad \text{in $(0, \tmax)$},
  \end{align*}
  where we have used that $(1-k)\intom v_t = 0$ by \eqref{eq:eps_reg:reg}.
  This yields \eqref{eq:energy:est} upon an integration in time.
\end{proof}

As already observed in, e.g., \cite[p.~1059]{CieslakWinklerFinitetimeBlowupQuasilinear2008} or \cite[Lemma~5]{CieslakLaurencotFiniteTimeBlowup2010},
in the considered one-dimensional setting, $\mc F$ can be estimated from below along solutions and hence the dissipation rate is integrable in time.
\begin{lemma}\label{lm:vt_l2_spacetime}
  Let $(\tau,k)\in \{(0,0),(0,1),(1,1)\}$ and $m \in \R$, suppose (\ref{initial_data})
  and denote the solution of \eqref{system} given by Proposition~\ref{prop:eps_reg} by $(u, v)$.
  For all $t_0 \in (0, \tmax)$, there then exists $C > 0$ such that
  \begin{align}\label{eq:dissipation}
    \int_{t_0}^{\tmax} \intom (\tau v_t)^2 \le C.
  \end{align}
\end{lemma}
\begin{proof}
  Invoking the one-dimensional embedding $\sob11 \embed \leb\infty$, \eqref{eq:u_mass} and Young's inequality,
  we see that
  \begin{align}
  \label{eq:energy:uv}
          \int_0^1 uv &\le  \norm[L^\infty(0,1)]{v} \int_0^1 u
    \le   M\norm[L^\infty((0,1))]{v}
    \le   M \norm[W^{1,1}((0,1))]{v}
    \le   M \left(\int_0^1 |v_x| + \int_0^1 |v|\right) \nn \\
    &\le  M\int_0^1 |v_x| + M \max\{M, \norm[L^1((0,1))]{v_0}\}
    \le   \frac{1}{2} \int_0^1 v_x^2 + c_1
    \qquad \text{in $(0, \tmax)$},
  \end{align}
  where $c_1 \defs \frac{M^2}{2} + M \max\{M, \norm[L^1((0,1))]{v_0}\}$.
  Since $G$ is nonnegative as the inner integrand in \eqref{eq:energy:def_G} is, plugging \eqref{eq:energy:uv} into \eqref{eq:energy:def_F} yields
  \begin{align*}
    \mathcal F(u,v) \ge - c_1
    \qquad \text{in $(0, \tmax)$},
  \end{align*}
  whence \eqref{eq:dissipation} follows from \eqref{eq:energy:est}.
\end{proof}

We next aim to control the evolution of the functional $\intom \psi(u)$, where $\psi(u) \sim u^{-m}$ for $m < -1$.
Related functionals have been considered in Lemma~\ref{lm:eqs_reg_lp},
but arguing as in the proof of Lemma~\ref{lm:eqs_reg_lp} would require applying maximal Sobolev regularity estimates (cf.\ Lemma~\ref{lm:max_reg})
with the integrability exponent $1$, which is not permitted.
Instead, we shall rely on the pointwise estimate of $v_x$ derived in Lemma~\ref{lm:vx_linfty}(ii).
This leads to an additional summand $\intom (\tau v_t)^2$,
which eventually will be treated by Lemma~\ref{lm:vt_l2_spacetime}.

We also note that for the following proof, the embedding $\sob11 \embed \leb\infty$ is crucial
(unlike in the proof of Lemma~\ref{lm:vt_l2_spacetime} above, where we could also have used $\sob1p \embed \leb \infty$ for any $p \in (1, 2)$ instead).
\begin{lemma}\label{lm:ddt_psi_u}
  Let $(\tau, k) \in \{(0, 0), (0, 1), (1, 1)\}$ and $m \in \R$,
  and suppose \eqref{initial_data}.
  Then there exist $C > 0$ and
  \begin{align}\label{eq:ddt_psi_u:f_cond}
    f \in C^0([0, \infty)) \quad \text{with} \quad \lim_{s \to \infty} f(s) = \infty
  \end{align}
  such that the solution $(u, v)$ given by Proposition~\ref{prop:eps_reg} fulfils
  \begin{align}\label{eq:ddt_psi_u:statement}
         \ddt \int_0^1 \psi(u)
    &\le -f\left(\int_0^1 \psi(u)\right) + C + C \int_0^1 (\tau v_t)^2
    \qquad \text{in $(0, \tmax)$},
  \end{align}
  where
  \begin{align}\label{eq:ddt_psi_u:def_psi}
    \psi(s) \defs \int_0^s \int_0^\sigma (\rho + 1)^{-m-2} \drho \dsigma \qquad \text{for $s \ge 0$}.
  \end{align}
\end{lemma}
\begin{proof}
  A direct computation yields
  \begin{align}\label{eq:ddt_psi_u:ddt}
          \ddt \int_0^1 \psi(u)
    &=    \int_0^1 \big( (u+1)^{m-1} u_x \big)_x \psi'(u)
          - \int_0^1 \big( u (u+1)^m v_x \big)_x \psi'(u) \notag \\
    &=    - \int_0^1 (u+1)^{-3} u_x^2
          + \int_0^1 u (u+1)^{-2} u_x v_x \notag \\
    &\le  - \frac12 \int_0^1 (u+1)^{-3} u_x^2
          + \frac12 \int_0^1 u^2 (u+1)^{-1} v_x^2 \notag \\
    &\le  - \frac12 \int_0^1 (u+1)^{-3} u_x^2
          + \frac{M c_1}2 \left( \int_0^1 (\tau v_t)^2 + 1 \right)
    \qquad \text{in $(0, \tmax)$},
  \end{align}
  where $c_1 > 0$ denotes the constant given by Lemma~\ref{lm:vx_linfty}(ii).
  If $\psi$ is bounded, this already implies \eqref{eq:ddt_psi_u:statement} for $f \defs \operatorname{id}$ and $C \defs \|\psi\|_{L^\infty((0, \infty))} + \frac{M c_1}{2}$.
  Accordingly, we now assume that $\psi$ is unbounded.
  Since $\sob11 \embed \leb\infty$ (with embedding constant $1$), we have
  \begin{align*}
          \|\ln(u+1)\|_{\leb\infty}
    &\le  \int_0^1 \big|\big(\ln(u + 1)\big)_x\big| + \int_0^1 \ln(u + 1) \\
    &\le  \int_0^1 (u+1)^{-1} |u_x| + \int_0^1 (u + 1) \\
    &\le  \int_0^1 (u+1)^{-3} u_x^2 + 2 \int_0^1 (u + 1) \\
    &\le  \int_0^1 (u+1)^{-3} u_x^2 + c_2
    \qquad \text{in $(0, \tmax)$},
  \end{align*}
  where $c_2 \defs 2(M + 1)$.
  Since $\psi' > 0$ in $(0, \infty)$ and as $\psi$ is unbounded, $\psi$ is a bijection on $[0, \infty)$ with increasing inverse.
  Thus,
  \begin{align*}
    f \colon [0, \infty) \to [0, \infty), \quad
    f(s) = \frac{1}{2} \ln\big(\max\{\psi^{-1}(s), 1\}\big),
  \end{align*}
  is also increasing and fulfils \eqref{eq:ddt_psi_u:f_cond}.
  Moreover,
  \begin{align*}
        f\left( \int_0^1 \psi(u) \right)
    &\le f\left( \int_0^1 \psi\left( \ure^{\ln(u+1)} \right) \right)
    \le f\left( \psi\left( \ure^{\|\ln(u+1)\|_{\leb\infty}} \right)\right) \\
    &=  \frac{1}{2} \ln\left( \max\left\{\ure^{\|\ln(u+1)\|_{\leb\infty}}, 1\right\}\right)
    =   \frac{1}{2} \|\ln(u+1)\|_{\leb\infty} \\
    &\le \frac12 \int_0^1 (u+1)^{-3} u_x^2 + \frac{c_2}{2}
    \qquad \text{in $(0, \tmax)$}.
  \end{align*}
  In combination with \eqref{eq:ddt_psi_u:ddt}, this yields
  \begin{align*}
        \ddt \int_0^1 \psi(u)
    &\le -f\left(\int_0^1 \psi(u)\right) + \frac{c_2 + M c_1}{2} + \frac{Mc_1}{2} \int_0^1 (\tau v_t)^2
    \qquad \text{in $(0, \tmax$)}
  \end{align*}
  and hence \eqref{eq:ddt_psi_u:statement} upon an evident choice of $C > 0$.
\end{proof}

In order to derive boundedness of $\int_0^1 \psi(u)$ from Lemma~\ref{lm:ddt_psi_u} and Lemma~\ref{lm:vt_l2_spacetime}, we will make use of the following elementary ODE lemma.
\begin{lemma}\label{lm:ode}
  Let $0 \le t_0 < T \le \infty$, suppose that $f \in C^0(\R)$, $L > 0$, $g \in C^0((t_0, T))$ and $y \in C^0([t_0, T)) \cap C^1((t_0, T))$ are such that
  \begin{align}\label{eq:ode:ass_ode}
    y' \le -f(y) + L + g
    \qquad \text{in $(t_0, T)$},
  \end{align}
  and assume that there are $K_f, K_g > 0$ such that
  \begin{align}\label{eq:ode:ass_f}
    f(s) &\ge L \qquad \text{for all $s \ge K_f$}
  \end{align}
  and
  \begin{align}\label{eq:ode:ass_g}
    \int_{t_0}^T g(s) \ds &\le K_g.
  \end{align}
  Then
  \begin{align}\label{eq:ode:statement}
    y(t) \le \max\{y(t_0), K_f\} + K_g
    \qquad \text{for all $t \in (t_0, T)$}.
  \end{align}
\end{lemma}
\begin{proof}
  Since $y$ is continuous, the set $A \defs \{\, t \in (t_0, \infty) \mid y(t) > \max\{y(t_0), K_f\} \sfed c_1\,\}$ is open.
  Then $f(y) \ge L$ in $A$ by \eqref{eq:ode:ass_f} and hence \eqref{eq:ode:ass_ode} implies
  \begin{align}\label{eq:ode:ass_ode2}
    y' \le -f(y) + L + g \le g
    \qquad \text{in $A$}.
  \end{align}
  Let $I$ be a connected component of $A$, then $I$ is an open interval.
  By continuity of $y$ and as $y(t_0) \le c_1$, we have $y(\inf I) \le c_1$, so that an integration of \eqref{eq:ode:ass_ode2} in time and applying \eqref{eq:ode:ass_g} shows
  \begin{align*}
        y(t)
    \le y(\inf I) + \int_{\inf I}^t g(s) \ds
    \le c_1 + \int_{t_0}^\infty g(s) \ds
    \le c_1 + K_g
    \qquad \text{for all $t \in I$}.
  \end{align*}
  That is, \eqref{eq:ode:statement} also holds for all $t \in A$.
\end{proof}

If $m \le -1$, the lemmata above imply an $L\log L$ bound for $u$, which renders Proposition~\ref{prop:eps_reg} applicable.
\begin{lemma}\label{lm:ulnu}
  Let $(\tau, k) \in \{(0, 0), (0, 1), (1, 1)\}$ and $m \le -1$,
  and suppose \eqref{initial_data}.
  For each $t_0 \in (0, \tmax)$, there then exists $C > 0$ such that
  \begin{align}\label{eq:ulnu:statement}
    \int_0^1 (u+1) \ln(u+1) \le C
    \qquad \text{in $(t_0, \tmax)$.}
  \end{align}
\end{lemma}
\begin{proof}
  We let $\psi$ be as in \eqref{eq:ddt_psi_u:def_psi}
  and denote the function and constant given by Lemma~\ref{lm:ddt_psi_u} by $f$ and $c_1$, respectively.
  Due to Lemma~\ref{lm:ddt_psi_u} and Lemma~\ref{lm:vt_l2_spacetime}, an application of Lemma~\ref{lm:ode} to this choice of $f$ (extended by $f(0)$ on $(-\infty, 0)$), $L = c_1$, $g = c_1 \int_0^1 (\tau v_t)^2$ and $T = \tmax$ yields
  \begin{align*}
    \sup_{t \in (t_0, \tmax)} \int_0^1 \psi(u(\cdot, t)) \le c_2
  \end{align*}
  for some $c_2 > 0$.
  This entails \eqref{eq:ulnu:statement} for $C \defs c_2 + M$, since the assumption $m \le -1$ implies
  \begin{align*}
        \psi(s)
    =   \int_0^s \int_0^\sigma (\rho + 1)^{-m-2} \drho \dsigma
    \ge \int_0^s \int_0^\sigma (\rho + 1)^{-1} \drho \dsigma
    =   \int_0^s \ln(\sigma + 1) \dsigma
    =   (s+1)\ln(s + 1) - s
  \end{align*}
  for all $s \ge 0$.
\end{proof}

\begin{proof}[Proof of Theorem~\ref{th:gb_pp}]
  Given such a solution $u$, we fix an arbitrary $t_0 \in (0, \tmax)$.
  Then $u$ is bounded in $\Ombar \times [0, t_0]$, which combined with Lemma~\ref{lm:ulnu} shows boundedness of $\intom (u+1) \ln(u+1)$ in $(0, \tmax)$.
  Therefore, the claim follows from Proposition~\ref{prop:eps_reg_alt_cond}(ii) and Proposition~\ref{prop:eps_reg}.
\end{proof}

\appendix
\section{Appendix}\label{appendix}
We have postponed the proof of Proposition¨\ref{lm:ex_w_hoelder_u_not_equiint}, which asserts that the condition in Proposition~\ref{prop:eps_reg_alt_cond}(i) is strictly stronger than the one in Proposition~\ref{prop:eps_reg_alt_cond}(iii).
\begin{proof}[Proof of Proposition~\ref{lm:ex_w_hoelder_u_not_equiint}]
  As to (i), we let $k \in \N$ and set
  \begin{align*}
    u_k \defs \sum_{j=0}^{k-1} u_{k, j}, \quad \text{where} \quad
    u_{k, j} \defs k \mathds 1_{(a_{k, j}, b_{k, j})}, \quad
    a_{k, j} \defs \frac{j}{k}, \quad
    b_{k, j} \defs \frac{j}{k} + \frac{1}{k^2}.
  \end{align*}
  The intervals $(a_{k, j}, b_{k, j})$ are pairwise disjoint and contained in $(0, 1)$.
  Since $|(a_{k, j}, b_{k, j})| = k^{-2}$, 
  \begin{align*}
    \int_0^1 u_k = \frac{k \cdot k}{k^2} = 1
    \quad \text{and} \quad
    |\operatorname{ess\,supp} u_k| = \frac{k}{k^2} = \frac1k.
  \end{align*}
  This implies \eqref{eq:ex_w_hoelder_u_not_equiint:not_equi}
  and that the function $w_k$, defined by $w_k(x) = \int_0^x u_k(\rho) \drho$ for $x \in [0, 1]$, belongs to $W^{1, 1}((0, 1))$.
  For all $j \in \{0, \dots, k-1\}$ and all $0 \le c \le d \le k^{-2}$ (implying that $(d-c)^\frac12 \le k^{-1}$), we have
  \begin{align*}
        w_k(a_{k, j} + d) - w_k(a_{k, j} + c)
    =   \int_{a_j + c}^{a_j + d} u_{k, j}
    =   k (d-c)
    \le (d-c)^\frac12,
  \end{align*}
  while for all $j \in \{0, \dots, k-1\}$ and all $i \in \{0, k-j\}$,
  \begin{align*}
        w_k(a_{k, j+i}) - w_k(a_{k, j})
    =   i \int_0^1 u_{k, 1}
    =   \frac{i}{k}
    =   \frac{i^\frac12}{k^\frac12} (a_{k, j+i} - a_{k, j})^\frac12
    \le (a_{k, j+i} - a_{k, j})^\frac12.
  \end{align*}
  In combination, for all $j \in \{0, \dots, k-1\}$, $i \in \{1, k-j-1\}$ and $c, d \in [0, k^{-2}]$, this yields
  \begin{align*}
    &\pe  w_k(a_{k, j+i} + d) - w_k(a_{k, j} + c) \\
    &=    w_k(a_{k, j+i} + d) - w_k(a_{k, j+i})
          + w_k(a_{k, j+i}) - w_k(a_{k, j+1})
          + w_k(b_{k, j}) - w_k(a_{k, j} + c) \\
    &\le  (a_{k, j+i} + d - a_{k, j+i})^\frac12
          + (a_{k, j+i} - a_{k, j+1})^\frac12
          + \big(b_{k, j} - (a_{k, j} + c)\big)^\frac12 \\
    &\le  \sqrt3 \big((a_{k, j+i} + d) - (a_{k, j} + c)\big)^\frac12.
  \end{align*}
  Since $w$ is constant outside of the intervals $(a_{k, j}, b_{k, j})$,
  we conclude $|w_k(x) - w_k(y)| \le \sqrt3 |x-y|^\frac12$ for all $x, y \in [0, 1]$.
  As moreover $\|w_k\|_{\con0} = \|u_k\|_{\leb1} = 1$, also the first assertion in \eqref{eq:ex_w_hoelder_u_not_equiint:hoelder} holds. That is, (i) is proven.

  The functions $w_k$ from (i) rapidly yield examples for (ii).
  For instance, for any $\zeta \in C^0([0, 1); [0, 1])$ satisfying $\zeta(1-2^{-2k}) = 0$ and $\zeta(1-2^{-2k-1}) = 1$ for all $k \in \N_0$, the function
  \[
    w \colon [0, 1) \to \R, \quad w(x, t) = w_k(x) \zeta(t) \text{ for $x \in [1-2^{-2k}, 1-2^{-2k-2})$, $k \in \N_0$ and $t \in [0, 1)$},
  \]
  can easily be seen to satisfy all required properties.
\end{proof}

\section*{Acknowledgments}
{The first author acknowledges support of the Natural Science Foundation of Shanghai within the project 23ZR1400100 and the Fundamental Research Funds for the Central Universities.}



\addcontentsline{toc}{section}{References}

\footnotesize
\begin{thebibliography}{10}
\setlength{\itemsep}{0.2pt}

\bibitem{AmannNonhomogeneousLinearQuasilinear1993}
\textsc{Amann, H.}:
\newblock {\em Nonhomogeneous linear and quasilinear elliptic and
  pa\-ra\-bo\-lic boundary value problems}.
\newblock In \textsc{Schmeisser, H.} and \textsc{Triebel, H.}, editors, {\em
  Function {{Spaces}}, {{Differential Operators}} and {{Nonlinear Analysis}}},
  \href{https://doi.org/10.1007/978-3-663-11336-2_1}{pages 9--126}.
  Vieweg+Teubner Verlag, Wiesbaden, 1993.
\newblock

\bibitem{bellomo2017finite}
\textsc{Bellomo, N.} and \textsc{Winkler, M.}:
\newblock {\em Finite-time blow-up in a degenerate chemotaxis system with flux
  limitation}.
\newblock Transactions of the American Mathematical Society, Series B,
  \href{https://doi.org/10.1090/btran/17}{4(2):31--67}, 2017.
\newblock

\bibitem{BieganowskiEtAlBoundednessSolutionsCritical2019}
\textsc{Bieganowski, B.}, \textsc{Cie{\'s}lak, T.}, \textsc{Fujie, K.}, and
  \textsc{Senba, T.}:
\newblock {\em Boundedness of solutions to the critical fully parabolic
  quasilinear one-dimensional {{Keller-Segel}} system}.
\newblock Math. Nachr.,
  \href{https://doi.org/10.1002/mana.201800175}{292(4):724--732}, 2019.
\newblock

\bibitem{BlackEtAlRelaxedParameterConditions2021}
\textsc{Black, T.}, \textsc{Fuest, M.}, and \textsc{Lankeit, J.}:
\newblock {\em Relaxed parameter conditions for chemotactic collapse in
  logistic-type parabolic--elliptic {{Keller--Segel}} systems}.
\newblock Z. F{\"u}r Angew. Math. Phys.,
  \href{https://doi.org/10.1007/s00033-021-01524-8}{72(3):Art.~96}, 2021.
\newblock

\bibitem{BlanchetEtAlTwodimensionalKellerSegelModel2006}
\textsc{Blanchet, A.}, \textsc{Dolbeault, J.}, and \textsc{Perthame, B.}:
\newblock {\em Two-dimensional {{Keller-Segel}} model: optimal critical mass
  and qualitative properties of the solutions}.
\newblock Electron. J. Differ. Equ., pages No. 44, 32 pages, 2006.

\bibitem{caffarelli1982partial}
\textsc{Caffarelli, L.}, \textsc{Kohn, R.}, and \textsc{Nirenberg, L.}:
\newblock {\em Partial regularity of suitable weak solutions of the
  {{Navier-Stokes}} equations}.
\newblock Comm. Pure Appl. Math.,
  \href{https://doi.org/10.1002/cpa.3160350604}{35(6):771--831}, 1982.
\newblock

\bibitem{CaoGlobalSolutionsChemotaxis2017}
\textsc{Cao, X.}:
\newblock {\em Global solutions of some chemotaxis systems}.
\newblock PhD thesis, Universit\"at Paderborn, Paderborn, 2017.

\bibitem{CaoLargeTimeBehavior2017}
\textsc{Cao, X.}:
\newblock {\em Large time behavior in the logistic {{Keller-Segel}} model via
  maximal {{Sobolev}} regularity}.
\newblock Discrete Contin. Dyn. Syst. - Ser. B,
  \href{https://doi.org/10.3934/dcdsb.2017141}{22(9):3369--3378}, 2017.
\newblock

\bibitem{CaoFuestFinitetimeBlowupFully2025}
\textsc{Cao, X.} and \textsc{Fuest, M.}:
\newblock {\em Finite-time blow-up in fully parabolic quasilinear
  {{Keller}}--{{Segel}} systems with supercritical exponents}.
\newblock Calc. Var. Partial Differ. Equ.,
  \href{https://doi.org/10.1007/s00526-025-02944-4}{64:Art. 89}, 2025.
\newblock

\bibitem{CaoGaoCriticalMassQuasilinear2023}
\textsc{Cao, X.} and \textsc{Gao, X.}:
\newblock {\em Critical mass in a quasilinear parabolic-elliptic
  {{Keller-Segel}} model}.
\newblock J. Differ. Equ.,
  \href{https://doi.org/10.1016/j.jde.2023.03.005}{361:449--471}, 2023.
\newblock

\bibitem{caotao2021boundedness}
\textsc{Cao, X.} and \textsc{Tao, Y.}:
\newblock {\em Boundedness and stabilization enforced by mild saturation of
  taxis in a producer-scrounger model}.
\newblock Nonlinear Anal. Real World Appl.,
  \href{https://doi.org/10.1016/j.nonrwa.2020.103189}{57:Paper No. 103189, 24
  pages}, 2021.
\newblock

\bibitem{cieslak2018no}
\textsc{Cie{\'s}lak, T.} and \textsc{Fujie, K.}:
\newblock {\em No critical nonlinear diffusion in {{1D}} quasilinear fully
  parabolic chemotaxis system}.
\newblock Proc. Am. Math. Soc.,
  \href{https://doi.org/10.1090/proc/13939}{146(6):2529--2540}, 2018.
\newblock

\bibitem{CieslakFujieGlobalExistence1D2020}
\textsc{Cie{\'s}lak, T.} and \textsc{Fujie, K.}:
\newblock {\em Global existence in the {{1D}} quasilinear parabolic-elliptic
  chemotaxis system with critical nonlinearity}.
\newblock Discrete Contin. Dyn. Syst. Ser. S,
  \href{https://doi.org/10.3934/dcdss.2020009}{13(2):165--176}, 2020.
\newblock

\bibitem{CieslakEtAlNonlinearFisherInformation2025}
\textsc{Cie{\'s}lak, T.}, \textsc{Fujie, K.}, and \textsc{Hosono, T.}:
\newblock {\em Nonlinear {{Fisher}} information, corresponding functional
  inequalities and applications}.
\newblock Preprint, \href {https://arxiv.org/abs/2509.01475}
  {{arXiv:2509.01475}}, 2025.

\bibitem{CieslakLaurencotFiniteTimeBlowup2010}
\textsc{Cie{\'s}lak, T.} and \textsc{Lauren{\c c}ot, {\relax Ph}.}:
\newblock {\em Finite time blow-up for a one-dimensional quasilinear
  parabolic--parabolic chemotaxis system}.
\newblock Ann. Inst. Henri Poincar{\'e} C Anal. Non Lin{\'e}aire,
  \href{https://doi.org/10.1016/j.anihpc.2009.11.016}{27(1):437--446}, 2010.
\newblock

\bibitem{cieslak2010looking}
\textsc{Cie{\'s}lak, T.} and \textsc{Lauren{\c c}ot, {\relax Ph}.}:
\newblock {\em Looking for critical nonlinearity in the one-dimensional
  quasilinear {{Smoluchowski-Poisson}} system}.
\newblock Discrete Contin. Dyn. Syst.,
  \href{https://doi.org/10.3934/dcds.2010.26.417}{26(2):417--430}, 2010.
\newblock

\bibitem{CieslakStinnerFinitetimeBlowupGlobalintime2012}
\textsc{Cie{\'s}lak, T.} and \textsc{Stinner, {\relax Ch}.}:
\newblock {\em Finite-time blowup and global-in-time unbounded solutions to a
  parabolic--parabolic quasilinear {{Keller}}--{{Segel}} system in higher
  dimensions}.
\newblock J. Differ. Equ.,
  \href{https://doi.org/10.1016/j.jde.2012.01.045}{252(10):5832--5851}, 2012.
\newblock

\bibitem{CieslakStinnerFiniteTimeBlowupSupercritical2014}
\textsc{Cie{\'s}lak, T.} and \textsc{Stinner, {\relax Ch}.}:
\newblock {\em Finite-time blowup in a supercritical quasilinear
  parabolic-parabolic {{Keller}}--{{Segel}} system in dimension 2}.
\newblock Acta Appl. Math.,
  \href{https://doi.org/10.1007/s10440-013-9832-5}{129(1):135--146}, 2014.
\newblock

\bibitem{CieslakStinnerNewCriticalExponents2015}
\textsc{Cie{\'s}lak, T.} and \textsc{Stinner, {\relax Ch}.}:
\newblock {\em New critical exponents in a fully parabolic quasilinear
  {{Keller}}--{{Segel}} system and applications to volume filling models}.
\newblock J. Differ. Equ.,
  \href{https://doi.org/10.1016/j.jde.2014.12.004}{258(6):2080--2113}, 2015.
\newblock

\bibitem{CieslakWinklerFinitetimeBlowupQuasilinear2008}
\textsc{Cie{\'s}lak, T.} and \textsc{Winkler, M.}:
\newblock {\em Finite-time blow-up in a quasilinear system of chemotaxis}.
\newblock Nonlinearity,
  \href{https://doi.org/10.1088/0951-7715/21/5/009}{21(5):1057--1076}, 2008.
\newblock

\bibitem{DiBenedettoDegenerateParabolicEquations1993}
\textsc{DiBenedetto, E.}:
\newblock {\em Degenerate {{Parabolic Equations}}}.
\newblock Universitext. Springer New York, New York, NY, 1993.
\newblock

\bibitem{DingWinklerRadialBlowupQuasilinear2024}
\textsc{Ding, M.} and \textsc{Winkler, M.}:
\newblock {\em Radial blow-up in quasilinear {{Keller-Segel}} systems:
  Approaching the full picture}.
\newblock Nonlinearity, 37(12):Paper No. 125006, 34, 2024.

\bibitem{FriedmanPartialDifferentialEquations1976}
\textsc{Friedman, A.}:
\newblock {\em Partial differential equations}.
\newblock R. E. Krieger Pub. Co, Huntington, N.Y, 1976.

\bibitem{FuestApproachingOptimalityBlowup2021}
\textsc{Fuest, M.}:
\newblock {\em Approaching optimality in blow-up results for
  {{Keller}}--{{Segel}} systems with logistic-type dampening}.
\newblock Nonlinear Differ. Equ. Appl. NoDEA,
  \href{https://doi.org/10.1007/s00030-021-00677-9}{28(2):Art.~16}, 2021.
\newblock

\bibitem{FuestLankeitCornersCollapseSimple2023}
\textsc{Fuest, M.} and \textsc{Lankeit, J.}:
\newblock {\em Corners and collapse: {{Some}} simple observations concerning
  critical masses and boundary blow-up in the fully parabolic
  {{Keller}}--{{Segel}} system}.
\newblock Appl. Math. Lett.,
  \href{https://doi.org/10.1016/j.aml.2023.108788}{146:Article 108788}, 2023.
\newblock

\bibitem{FuestEtAlBlowupFullyParabolic2026}
\textsc{Fuest, M.}, \textsc{Lankeit, J.}, and \textsc{Mizukami, M.}:
\newblock {\em Blow-up in the fully parabolic {{Keller}}--{{Segel}} system with
  logistic source}.
\newblock Preprint, 2026.

\bibitem{GajewskiZachariasGlobalBehaviourReactiondiffusion1998}
\textsc{Gajewski, H.} and \textsc{Zacharias, K.}:
\newblock {\em Global behaviour of a reaction-diffusion system modelling
  chemotaxis}.
\newblock Math. Nachrichten,
  \href{https://doi.org/10.1002/mana.19981950106}{195:77--114}, 1998.
\newblock

\bibitem{HieberPrussHeatKernelsMaximal1997}
\textsc{Hieber, M.} and \textsc{Pr{\"u}ss, J.}:
\newblock {\em Heat kernels and maximal ${L}^p$-${L}^q$ estimates for parabolic
  evolution equations}.
\newblock Commun. Partial Differ. Equ.,
  \href{https://doi.org/10.1080/03605309708821314}{22(9-10):1647--1669}, 1997.
\newblock

\bibitem{HillenPainterUserGuidePDE2009}
\textsc{Hillen, T.} and \textsc{Painter, K.~J.}:
\newblock {\em A user's guide to {{PDE}} models for chemotaxis}.
\newblock J. Math. Biol.,
  \href{https://doi.org/10.1007/s00285-008-0201-3}{58(1-2):183--217}, 2009.
\newblock

\bibitem{HorstmannWangBlowupChemotaxisModel2001}
\textsc{Horstmann, D.} and \textsc{Wang, G.}:
\newblock {\em Blow-up in a chemotaxis model without symmetry assumptions}.
\newblock Eur. J. Appl. Math.,
  \href{https://doi.org/10.1017/S0956792501004363}{12(02):159--177}, 2001.
\newblock

\bibitem{HorstmannWinklerBoundednessVsBlowup2005}
\textsc{Horstmann, D.} and \textsc{Winkler, M.}:
\newblock {\em Boundedness vs. blow-up in a chemotaxis system}.
\newblock J. Differ. Equ.,
  \href{https://doi.org/10.1016/j.jde.2004.10.022}{215(1):52--107}, 2005.
\newblock

\bibitem{IshidaEtAlBoundednessQuasilinearKeller2014}
\textsc{Ishida, S.}, \textsc{Seki, K.}, and \textsc{Yokota, T.}:
\newblock {\em Boundedness in quasilinear {{Keller}}--{{Segel}} systems of
  parabolic--parabolic type on non-convex bounded domains}.
\newblock J. Differ. Equ.,
  \href{https://doi.org/10.1016/j.jde.2014.01.028}{256(8):2993--3010}, 2014.
\newblock

\bibitem{JagerLuckhausExplosionsSolutionsSystem1992}
\textsc{J{\"a}ger, W.} and \textsc{Luckhaus, S.}:
\newblock {\em On explosions of solutions to a system of partial differential
  equations modelling chemotaxis}.
\newblock Trans. Am. Math. Soc.,
  \href{https://doi.org/10.2307/2153966}{329(2):819--824}, 1992.
\newblock

\bibitem{KellerSegelInitiationSlimeMold1970}
\textsc{Keller, E.~F.} and \textsc{Segel, L.~A.}:
\newblock {\em Initiation of slime mold aggregation viewed as an instability}.
\newblock J. Theor. Biol.,
  \href{https://doi.org/10.1016/0022-5193(70)90092-5}{26(3):399--415}, 1970.
\newblock

\bibitem{LankeitInfiniteTimeBlowup2020}
\textsc{Lankeit, J.}:
\newblock {\em Infinite time blow-up of many solutions to a general quasilinear
  parabolic-elliptic {{Keller-Segel}} system}.
\newblock Discrete Contin. Dyn. Syst. - S,
  \href{https://doi.org/10.3934/dcdss.2020013}{13(2):233--255}, 2020.
\newblock

\bibitem{LankeitWinklerFacingLowRegularity2019}
\textsc{Lankeit, J.} and \textsc{Winkler, M.}:
\newblock {\em Facing low regularity in chemotaxis systems}.
\newblock Jahresber. Dtsch. Math.-Ver.,
  \href{https://doi.org/10.1365/s13291-019-00210-z}{122:35--64}, 2019.
\newblock

\bibitem{LaurencotMizoguchiFiniteTimeBlowup2017}
\textsc{Lauren{\c c}ot, {\relax Ph}.} and \textsc{Mizoguchi, N.}:
\newblock {\em Finite time blowup for the parabolic--parabolic
  {{Keller}}--{{Segel}} system with critical diffusion}.
\newblock Ann. Inst. Henri Poincar{\'e} C Anal. Non Lin{\'e}aire,
  \href{https://doi.org/10.1016/j.anihpc.2015.11.002}{34(1):197--220}, 2017.
\newblock

\bibitem{MaoLiNote8pProblem2024}
\textsc{Mao, X.} and \textsc{Li, Y.}:
\newblock {\em A note on the {{8$\pi$}} problem of {{J\"ager-Luckhaus}}
  system}.
\newblock Preprint, \href {https://arxiv.org/abs/2405.06315}
  {{arXiv:2405.06315}}, 2024.

\bibitem{MizoguchiFinitetimeBlowupCauchy2020}
\textsc{Mizoguchi, N.}:
\newblock {\em Finite-time blowup in {{Cauchy}} problem of parabolic-parabolic
  chemotaxis system}.
\newblock J. Math. Pures Appl. (9),
  \href{https://doi.org/10.1016/j.matpur.2019.10.004}{136:203--238}, 2020.
\newblock

\bibitem{MizoguchiWinklerBlowupTwodimensionalParabolic}
\textsc{Mizoguchi, N.} and \textsc{Winkler, M.}:
\newblock {\em Blow-up in the two-dimensional parabolic {{Keller}}--{{Segel}}
  system}.
\newblock Preprint.

\bibitem{NagaiBlowupNonradialSolutions2001}
\textsc{Nagai, T.}:
\newblock {\em Blowup of nonradial solutions to parabolic-elliptic systems
  modeling chemotaxis in two-dimensional domains}.
\newblock J. Inequalities Appl.,
  \href{https://doi.org/10.1155/S1025583401000042}{6(1):37--55}, 2001.
\newblock

\bibitem{NagaiEtAlApplicationTrudingerMoser1997}
\textsc{Nagai, T.}, \textsc{Senba, T.}, and \textsc{Yoshida, K.}:
\newblock {\em Application of the {{Trudinger}}--{{Moser}} inequality to a
  parabolic system of chemotaxis}.
\newblock Funkc. Ekvacioj, 40:411--433, 1997.

\bibitem{PainterHillenVolumefillingQuorumsensingModels2002}
\textsc{Painter, K.} and \textsc{Hillen, T.}:
\newblock {\em Volume-filling and quorum-sensing in models for chemosensitive
  movement}.
\newblock Can. Appl. Math. Q., 10(4):501--544, 2002.

\bibitem{PorzioVespriHolderEstimatesLocal1993}
\textsc{Porzio, M.} and \textsc{Vespri, V.}:
\newblock {\em Holder estimates for local solutions of some doubly nonlinear
  degenerate parabolic equations}.
\newblock J. Differ. Equ.,
  \href{https://doi.org/10.1006/jdeq.1993.1045}{103(1):146--178}, 1993.
\newblock

\bibitem{SenbaSuzukiQuasilinearParabolicSystem2006}
\textsc{Senba, T.} and \textsc{Suzuki, T.}:
\newblock {\em A quasi-linear parabolic system of chemotaxis}.
\newblock Abstr. Appl. Anal.,
  \href{https://doi.org/10.1155/AAA/2006/23061}{2006:1--21}, 2006.
\newblock

\bibitem{sugiyama2009varepsilon}
\textsc{Sugiyama, Y.}:
\newblock {\em On $\varepsilon$-regularity theorem and asymptotic behaviors of
  solutions for {{Keller}}--{{Segel}} systems}.
\newblock Siam Journal On Mathematical Analysis,
  \href{https://doi.org/10.1137/080721078}{41(4):1664--1692}, 2009.
\newblock

\bibitem{sugiyama2010varepsilon}
\textsc{Sugiyama, Y.}:
\newblock {\em $\varepsilon$-regularity theorem and its application to the
  blow-up solutions of {{Keller--Segel}} systems in higher dimensions}.
\newblock Journal of Mathematical Analysis and Applications,
  \href{https://doi.org/10.1016/j.jmaa.2009.11.019}{364(1):51--70}, 2010.
\newblock

\bibitem{TaoWinklerBoundednessQuasilinearParabolic2012}
\textsc{Tao, Y.} and \textsc{Winkler, M.}:
\newblock {\em Boundedness in a quasilinear parabolic--parabolic
  {{Keller}}--{{Segel}} system with subcritical sensitivity}.
\newblock J. Differ. Equ.,
  \href{https://doi.org/10.1016/j.jde.2011.08.019}{252(1):692--715}, 2012.
\newblock

\bibitem{TelloWinklerChemotaxisSystemLogistic2007}
\textsc{Tello, J.~I.} and \textsc{Winkler, M.}:
\newblock {\em A chemotaxis system with logistic source}.
\newblock Commun. Partial Differ. Equ.,
  \href{https://doi.org/10.1080/03605300701319003}{32(6):849--877}, 2007.
\newblock

\bibitem{WinklerDoesVolumefillingEffect2009}
\textsc{Winkler, M.}:
\newblock {\em Does a `volume-filling effect' always prevent chemotactic
  collapse?}
\newblock Math. Methods Appl. Sci.,
  \href{https://doi.org/10.1002/mma.1146}{33(1):12--24}, 2009.
\newblock

\bibitem{WinklerCriticalBlowupExponent2018}
\textsc{Winkler, M.}:
\newblock {\em A critical blow-up exponent in a chemotaxis system with
  nonlinear signal production}.
\newblock Nonlinearity,
  \href{https://doi.org/10.1088/1361-6544/aaaa0e}{31(5):2031--2056}, 2018.
\newblock

\bibitem{WinklerGlobalClassicalSolvability2019}
\textsc{Winkler, M.}:
\newblock {\em Global classical solvability and generic infinite-time blow-up
  in quasilinear {{Keller}}--{{Segel}} systems with bounded sensitivities}.
\newblock J. Differ. Equ.,
  \href{https://doi.org/10.1016/j.jde.2018.12.019}{266(12):8034--8066}, 2019.
\newblock

\bibitem{WinklerHowUnstableSpatial2018}
\textsc{Winkler, M.}:
\newblock {\em How unstable is spatial homogeneity in {{Keller--Segel}}
  systems? {{A}} new critical mass phenomenon in two- and higher-dimensional
  parabolic--elliptic cases}.
\newblock Math. Ann.,
  \href{https://doi.org/10.1007/s00208-018-1722-8}{373:1237--1282}, 2019.
\newblock

\bibitem{winkler2022family}
\textsc{Winkler, M.}:
\newblock {\em A family of mass-critical {{Keller-Segel}} systems}.
\newblock Proc. Lond. Math. Soc. Third Ser.,
  \href{https://doi.org/10.1112/plms.12425}{124(2):133--181}, 2022.
\newblock

\bibitem{WinklerDjieBoundednessFinitetimeCollapse2010}
\textsc{Winkler, M.} and \textsc{Djie, K.~C.}:
\newblock {\em Boundedness and finite-time collapse in a chemotaxis system with
  volume-filling effect}.
\newblock Nonlinear Anal. Theory Methods Appl.,
  \href{https://doi.org/10.1016/j.na.2009.07.045}{72(2):1044--1064}, 2010.
\newblock

\end{thebibliography}

\end{document}